\documentclass[12pt]{amsart}
\usepackage[margin=1in]{geometry}                
\usepackage{graphicx}
\usepackage{amssymb}
\usepackage{amsthm}
\usepackage{mathrsfs}
\usepackage{caption}

\newtheorem{theorem}{Theorem}[section]
\newtheorem{proposition}[theorem]{Proposition}
\newtheorem{lemma}[theorem]{Lemma}
\newtheorem{definition}[theorem]{Definition}
\newtheorem{corollary}[theorem]{Corollary}
\newtheorem{conjecture}[theorem]{Conjecture}
\newtheorem{remark}{Remark}[section]
\newtheorem{question}{Question}

\numberwithin{equation}{section}

\newcommand{\CC}{\mathbb{C}}

\newcommand{\ZZ}{\mathbb{Z}}
\newcommand{\RR}{\mathbb{R}}
\newcommand{\NN}{\mathbb{N}}

\newcommand{\bbA}{\mathbb{A}}
\newcommand{\bbB}{\mathbb{B}}

\newcommand{\calp}{\mathcal{P}}
\newcommand{\calt}{\mathcal{T}}

\newcommand{\cale}{\mathcal{E}}

\newcommand{\calm}{\mathcal{M}}
\newcommand{\call}{\mathcal{L}}
\newcommand{\one}{\mathbf{1}}

\newcommand{\Div}{{\mathrm{Div}}}
\newcommand{\Span}{{\mathrm{Span}}}
\newcommand{\Bispan}{{\mathrm{Bispan}}}


\title{Combinatorial and harmonic-analytic methods for integer tilings} 
\author{Izabella {\L}aba and Itay Londner}

\date{\today}

\begin{document}

\begin{abstract}
A finite set of integers $A$ tiles the integers by translations if $\mathbb{Z}$ can be covered by pairwise disjoint translated copies of $A$. Restricting attention to one tiling period, we have $A\oplus B=\mathbb{Z}_M$ for some $M\in\mathbb{N}$ and $B\subset\mathbb{Z}$. This can also be stated in terms of cyclotomic divisibility of the mask polynomials $A(X)$ and $B(X)$ associated with $A$ and $B$.

In this article, we introduce a new approach to a systematic study of such tilings. Our main new tools are the box product, multiscale cuboids, and saturating sets, developed through a combination of harmonic-analytic and combinatorial methods. We provide new criteria for tiling and cyclotomic divisibility in terms of these concepts. As an application, we can determine whether a set $A$ containing certain configurations can tile a cyclic group $\mathbb{Z}_M$, or recover a tiling set based on partial information about it. We also develop tiling reductions where a given tiling can be replaced by one or more tilings with a simpler structure.
The tools introduced here are crucial in our proof in \cite{LaLo2} that all tilings of period $(pqr)^2$, where $p,q,r$ are distinct odd primes, satisfy a tiling condition proposed by Coven and Meyerowitz 
\cite{CM}. 
\end{abstract}

\maketitle

\tableofcontents



\section{Introduction}


A set $A\subset\ZZ$ {\em tiles the integers by translations} if $\ZZ$ can be covered by pairwise disjoint translates of $A$. Equivalently, there exists a set $T\subset\ZZ$ (the set of translations) such that every integer $n\in\ZZ$ can be represented uniquely as $n=a+t$ with $a\in A$ and $t\in T$. Throughout this
article, we assume that $A$ is finite and nonempty, and call it a {\em finite tile} if it tiles the integers.
Newman \cite{New} proved that any tiling of $\ZZ$ by a finite set $A$ must be periodic, i.e.
$T=B\oplus M\ZZ$ for some finite set $B\subset \ZZ$ such that $|A|\,|B|=M$.  Equivalently,
$A\oplus B$ is a factorization of the cyclic group $\ZZ_M$, with $B$ as the {\em tiling complement}.

We are interested in investigating the properties of finite tiles. While this is a natural and attractive question, surprisingly little has been known on this subject. 

Newman's proof provides a bound on the tiling period, $M\leq 2^{\max(A)-\min(A)}$. Thus, given a finite set $A\subset\ZZ$, the question of whether $A$ is a tile is at least in principle computationally decidable.
However, Newman's bound is exponential in diameter, and can therefore be very large even if $A$ has only a few elements. A more effective bound was proved recently by Greenfeld and Tao \cite{GT}.

Further important reductions and observations were made by 
Sands \cite{Sands}, Tijdeman \cite{Tij}, and Coven and Meyerowitz \cite{CM}. Sands's theorem on replacement of
factors \cite{Sands} states that if $A\oplus B=\ZZ_M$ and $M$ has at most two distinct prime divisors,
then at least one of $A$ and $B$ must be contained in a proper subgroup of $\ZZ_M$. The proof of this is based on a characterization of tiling pairs, also due to Sands, which we state here as Theorem \ref{thm-sands}.
Tijdeman \cite{Tij} proved that if a finite set $A$ tiles the integers, and if $r\in\NN$ is relatively prime to $|A|$, then $rA:=\{ra:a\in A\}$ also tiles $\ZZ$ with the same tiling complement. 
Coven and Meyerowitz \cite[Lemma 2.3]{CM} used this to prove that if a finite set $A$ tiles the integers, then it also tiles
$\ZZ_M$ for some $M$ which has the same prime factors as $|A|$.

For the last two decades, the state-of-the-art work on the subject was due to Coven and Meyerowitz \cite{CM}. In order to describe their main result, we need to introduce some notation, which we will also use throughout this article.
By translational invariance, we may assume that
$A,B\subset\{0,1,\dots\}$ and that $0\in A\cap B$. The {\it characteristic 
polynomials} (also known as {\it mask polynomials}) of $A$ and $B$ are
$$
A(X)=\sum_{a\in A}X^a,\ B(X)=\sum_{b\in B}X^b .
$$
Then $A\oplus B=\ZZ_M$ is equivalent to
\begin{equation}\label{poly-e1}
 A(X)B(X)=1+X+\dots+X^{M-1}\ \mod (X^M-1).
\end{equation}

Let $\Phi_s(X)$ be the $s$-th cyclotomic
polynomial, i.e., the unique monic, irreducible polynomial whose roots are the
primitive $s$-th roots of unity. Alternatively, $\Phi_s$ can be defined inductively via the identity
\begin{equation}\label{poly-e0}
X^n-1=\prod_{s|n}\Phi_s(X).
\end{equation}
In particular, (\ref{poly-e1}) is equivalent to
\begin{equation}\label{poly-e2}
  |A||B|=M\hbox{ and }\Phi_s(X)\, |\, A(X)B(X)\hbox{ for all }s|M,\ s\neq 1.
\end{equation}
Since $\Phi_s$ are irreducible, each $\Phi_s(X)$ with $s|M$ must divide at least one of $A(X)$ and $B(X)$.

Coven and Meyerowitz \cite{CM} proved the following theorem. 

\begin{theorem}\label{CM-thm} \cite{CM}
Let $S_A$ be the set of prime powers
$p^\alpha$ such that $\Phi_{p^\alpha}(X)$ divides $A(X)$.  Consider the following conditions.

\smallskip
{\it (T1) $A(1)=\prod_{s\in S_A}\Phi_s(1)$,}

\smallskip
{\it (T2) if $s_1,\dots,s_k\in S_A$ are powers of different
primes, then $\Phi_{s_1\dots s_k}(X)$ divides $A(X)$.}
\medskip

Then:

\begin{itemize}

\item if $A$ satisfies (T1), (T2), then $A$ tiles $\ZZ$;

\item  if $A$ tiles $\ZZ$ then (T1) holds;

\item if $A$ tiles $\ZZ$ and $|A|$ has at most two distinct prime factors,
then (T2) holds.
\end{itemize}

\end{theorem}

The condition (T1) is, essentially, a counting condition, and is relatively easy to prove. For sets $A\subset\ZZ$ such that $|A|$ is a prime power, (T1) is a necessary and sufficient condition for $A$ to be a tile \cite{New}. (In this case, (T2) is vacuous.)

The second condition (T2) is much deeper.  Coven and Meyerowitz \cite{CM} proved that if (T2) holds, then $A\oplus B^\flat =\ZZ_M$ is a tiling, where $M={\rm lcm}(S_A)$ and $B^\flat$ is an explicitly constructed and highly structured ``standard" tiling complement depending only on the prime power cyclotomic divisors of $A(X)$. We prove in Section \ref{standards} that having a tiling complement of this type is in fact equivalent to (T2). While this equivalence was not stated explicitly in \cite{CM}, it follows readily from the methods developed there.

The Coven-Meyerowitz proof of (T2) for all finite tiles with 2 distinct prime factors relies on the aforementioned structure and replacement theorems of Sands \cite{Sands} and Tijdeman \cite{Tij}. 
In \cite[Lemma 2.3]{CM}, the authors deduce from Tijdeman's theorem that if $A$ tiles the integers and $|A|$ has at most two distinct prime factors, then $A$ admits a tiling $A\oplus B=\ZZ_M$, where
$M$ has at most two distinct prime factors. By Sands's theorem,
one of $A$ and $B$ must then be contained in a proper subgroup of $\ZZ_M$.  Coven and
Meyerowitz use this to set up an inductive argument. 

A closer analysis of the Coven-Meyerowitz argument yields the same result in the case when $M=p_1^{n_1} \dots p_K^{n_K}$, where $p_1,\dots,p_K$ are distinct primes, $n_1,\dots,n_K\in\NN$ are arbitrary, and at most two of $p_1,\dots,p_K$ divide both $|A|$ and $|B|$. Essentially, any such case can be reduced to the two-prime case via Tijdeman's theorem and Lemma 2.3 of \cite{CM}, whereupon Theorem \ref{CM-thm} may be applied. We provide the details in Corollary \ref{almostsquarefree}. (See also \cite{Tao-blog}, \cite{dutkay-kraus}, \cite{shi}.)

The goal of the present article is to develop methods that can be used in the study of tilings $A\oplus B=\ZZ_M$, where $M$ is permitted to have three or more prime factors dividing both $|A|$ and $|B|$. Sands's factorization theorem does not hold in this case, with counterexamples in \cite{Sz}, \cite{LS}.
For the same reason, the Coven-Meyerowitz proof does not extend to such tilings. We emphasize that this is not just a technical issue. Tilings with three or more distinct prime factors dividing both $|A|$ and $|B|$
are genuinely different, and any comprehensive analysis of them must account for new phenomena that have no counterparts for two prime factors, such as Szab\'o's examples \cite{Sz}.

The simplest tilings that cannot be reduced to the 2-prime case using the methods of \cite{CM} are of the form $A\oplus B=\ZZ_M$, where $|A|=|B|=p_1p_2p_3$ and $p_1,p_2,p_3$ are distinct primes. 
In the follow-up paper \cite{LaLo2}, we use the methods developed here to resolve this case when $p_1,p_2,p_3$ are odd.

\begin{theorem}\label{T2-3primes-thm} \cite{LaLo2}
Let $M=p_i^2p_j^2p_k^2$, where $p_i,p_j,p_k$ are distinct odd primes. 
Assume that $A\oplus B=\ZZ_M$, with $|A|=|B|=p_ip_jp_k$. Then both $A$ and $B$ satisfy (T2).
\end{theorem}

While our complete proof of Theorem \ref{T2-3primes-thm} works only under the assumptions indicated, many of our tools, methods, and intermediate results apply to general tilings $A\oplus B=\ZZ_M$, raising the possibility of further extensions and improvements. We therefore chose to present them here in more generality, deferring the actual proof of Theorem \ref{T2-3primes-thm} to the paper \cite{LaLo2}, which is restricted to the 3-prime setting.

We begin with the notation and preliminaries in Section \ref{sec-prelim}. We identify $\ZZ_M=\ZZ_{p_1^{n_1} \dots p_K^{n_K}}$ with $\ZZ_{p_1^{n_1}}\oplus \dots \oplus\ZZ_{p_K^{n_K}}$, and use the induced coordinate system to identify the given tiling with a tiling of a multidimensional lattice. This allows a geometric viewpoint whereby we can describe the tiling in terms of objects such as lines, planes, or fibers (arithmetic progressions of maximal length on certain scales). We emphasize, however, that the problem under consideration is much more specific that the study of tilings of multidimensional lattices in general. It is important in our work that the different coordinate directions correspond to distinct primes.

In Section \ref{standards}, we present an alternative formulation of (T2) in terms of {\em standard tiling complements}. Roughly speaking, if $A\oplus B=\ZZ_M$ is a tiling, then $B$ satisfies (T2) if and only if its tiling complement $A$ can be replaced by a highly structured ``standard set" $A^\flat$ with the same prime power cyclotomic divisors as $A$. Such standard sets were already used in \cite{CM} to prove that (T1) and (T2) imply tiling. Here, we state the formal implication in the other direction. In this formulation, the condition (T2) can be viewed as a distant cousin of questions on replacement of factors in factorization of finite abelian groups (see \cite{Szabo-book} for an overview).

In Section \ref{sec-box-product}, we introduce one of our main tools, the {\em box product}. The idea comes from the unpublished paper \cite{GLW}, and our main harmonic-analytic identity, Theorem \ref{main-identity}, is in fact a reprise of \cite[Theorem 1]{GLW} with relatively minor modifications. We are, however, able to use it much more effectively. (We caution the reader that, while Theorem 1 in \cite{GLW} is correct, the proof of the main tiling result in \cite{GLW} contains an error that cannot be readily fixed with the methods of that paper.)

Our goal is to be able to start with an arbitrary tiling $A\oplus B=\ZZ_M$, and prove that either at least one of the sets $A$ or $B$ can be replaced directly by the corresponding standard tiling complement (which proves (T2) as indicated above), or else we can pass to tilings with a smaller period $N|M$ and apply an inductive argument. The machinery to do this is developed in Sections \ref{gen-cuboids}--\ref{sec-fibers}, and includes the following main ingredients.

Cuboids (Section \ref{gen-cuboids}) and fibering (Section \ref{sec-fibers}) are our main tools in determining cyclotomic divisibility and proving structural properties. Cuboids have been used previously in the literature in the context of vanishing sums of roots of unity \cite{Steinberger} and Fuglede's spectral set conjecture \cite{KMSV}. We often have to use both cuboids and fibering at several scales at the same time. In particular, we introduce ``multiscale" cuboids that correspond to divisibility by combinations of several cyclotomic polynomials.

In Section \ref{sec-reductions}, we discuss two reductions that allow us to pass to tilings with a smaller period, with the (T2) property preserved under the decomposition. We first review the subgroup reduction from \cite{CM}. Next, we introduce a ``slab reduction", which we believe to be new, and which covers many cases of interest that are not covered by the subgroup reduction. We also develop a criterion for this reduction to apply. A concrete example of this is provided in Corollary \ref{Afibered}.

While the subgroup reduction is sufficient to prove Theorem \ref{CM-thm}, tilings with 3 or more distinct prime factors include cases where such inductive arguments do not appear to be easily applicable. One well-known obstruction to an inductive approach is provided by Szab\'o-type examples \cite{Sz}. However, Szab\'o's examples are known to satisfy (T2). This was observed already by Coven and Meyerowitz \cite{CM}; see also \cite{dutkay-kraus} for an explicit analysis of a class of examples based on Szab\'o's idea. 

We do not know whether Szab\'o-type constructions are the only obstacle to an inductive proof of (T2) for all finite tiles. In \cite{LaLo2}, we prove that this is indeed true for classes of tilings that are broad enough to include all tilings of $\ZZ_M$ with $M=(p_1p_2p_3)^2$. 
The key new concept turns out to be {\em saturating sets} -- subsets of $A$ and $B$ that saturate appropriately chosen box products (Section \ref{sec-satsets}). Informally, if a tile $A$ contains geometric configurations that indicate lack of structure on a certain scale, we are able to use it to our advantage and locate highly structured configurations elsewhere in both $A$ and $B$. In particular, the less structure we have in one of the tiles, the more structured the other one is expected to be. In the $M=(p_1p_2p_3)^2$ case, we use this to prove that all tilings with ``unfibered grids" (see \cite{LaLo2} for the definition) must in fact come from Szab\'o-type constructions. With this established, we can prove (T2) for such tilings by reduction to standard tiling complements. 
The full argument is carried out in \cite{LaLo2}, but we also provide examples of this procedure here in Section \ref{shiftexamples}. 

In addition to applications to proving structural conditions such as (T2), 
we are also able to use saturating sets to identify sets $A\subset\ZZ_M$ that do {\em not} tile $\ZZ_M$ based on the presence of certain configurations. Results of this type include Lemma \ref{smallcube} and Proposition \ref{prop-twodirections}. 

In Section \ref{conjectures-section}, we discuss open questions and possible directions of study arising from our research so far.

Since the work of Coven and Meyerowitz, there has been essentially no progress on proving (T2), except for a few special cases of limited importance (such as \cite{KL}) and cases covered by Corollary \ref{almostsquarefree} (\cite{Tao-blog}, \cite{dutkay-kraus}, \cite{shi}). However, there have been significant recent developments on other questions related to tiling and cyclotomic divisibility. Notably,
Bhattacharya \cite{Bh} has established the periodic tiling conjecture in $\ZZ^2$, with a quantitative version due to Greenfeld and Tao \cite{GT}. In a continuous setting, there has been recent work on tilings of the real line by a function \cite{KoLev}.

There is an important connection between the Coven-Meyerowitz tiling conditions and Fuglede's spectral set conjecture \cite{Fug}. The conjecture, dating back to the 1970s, states that a set $\Omega\subset \RR^n$ of positive $n$-dimensional Lebesgue measure tiles $\RR^n$ by translations if and only if the space $L^2(\Omega)$ admits an orthogonal basis of exponential functions. A set with the latter property is called a {\em spectral set}. While the question originated in functional analysis, it has intriguing connections to many other areas of mathematics, from convex geometry to wavelets, oscillatory integral estimates, and number theory.
The conjecture is now known to be false in dimensions 3 and higher \cite{tao-fuglede}, \cite{KM}, \cite{KM2}, \cite{FMM}, \cite{matolcsi}, \cite{FR}. Nonetheless, many
important cases remain open and continue to attract attention. Iosevich, Katz and Tao \cite{IKT} proved in 2003 that Fuglede's conjecture holds for convex sets in $\RR^2$; an analogous result in higher dimensions was proved only recently, by Greenfeld and Lev \cite{GL} for $n=3$, and by Lev and Matolcsi \cite{LM} for general $n$. There has also been extensive work on the finite abelian group analogue of the conjecture \cite{IMP}, \cite{MK}, \cite{FMV}, \cite{KMSV}, \cite{KMSV2}, \cite{KS}, \cite{M}, \cite{shi}, \cite{shi2}, \cite{somlai}, \cite{FKS}, \cite{zhang}.

Combined with a sequence of results in \cite{LW1}, \cite{LW2}, \cite{L},
proving (T2) for all finite integer tiles would resolve the ``tiling implies spectrum" part of Fuglede's spectral set conjecture for all compact tiles of the real line in dimension 1. Additionally, Dutkay and Lai \cite{dutkay-lai} proved that if a similar property could be established for spectral sets, then this would also resolve the converse part of the conjecture for compact sets in $\RR$. While proving (T2) for a more narrow class of integer tiles does not have that implication, it still establishes one direction of Fuglede's conjecture for that class of tiles in the finite group setting, as well as for sets $E=\bigcup_{a\in A} [a,a+1] \subset\RR$, where $A\subset\ZZ$ is an integer tile in the permitted class \cite{L}.


\section{Notation and preliminaries}\label{sec-prelim}


\subsection{Multisets and mask polynomials}

Let $M\geq 2$ be a fixed integer. Usually, we will work in either $\ZZ_M$ or in $\ZZ_N$ for some $N|M$. 
In the context of the tiling problem, we reserve $M$
for the tiling period and $N$ for its divisors. We also reserve $K$ for the number of the distinct
prime divisors of $M$, and use $p_1,\dots,p_K$ to denote those divisors, so that
$$M=\prod_{i=1}^K p_i^{n_i},$$
where $p_1,\dots,p_K$ are distinct primes and $n_1,\dots,n_K\in\NN$. We fix this notation and use it throughout the rest of the article. For a prime $p$, an integer $m$, and a nonnegative integer $\alpha$, we will say that $p^\alpha\parallel m$ if $p^\alpha|m$ but $p^{\alpha+1}\nmid m$.

We use $A(X)$, $B(X)$, etc. to denote polynomials modulo $X^M-1$ with integer coefficients. 
If $A(X)$ is such a polynomial, we define its weight function $w_A:\ZZ_M\to \ZZ$ so that $w_A(a)$ is the
coefficient of $X^a$ in $A(X)$. Thus $A(X)=\sum_{a\in\ZZ_M} w_A(a) X^a$. If $A$ has 0-1 coefficients, then
$w_A$ is the characteristic function of a set $A\subset \ZZ_M$. However, we will also consider polynomials with integer coefficients not necessarily equal to 0 or 1. In that case, $A(X)$ will correspond to a weighted multiset in $\ZZ_M$, 
which we will also denote by $A$,
with weights $w_A(a)$ assigned to each $a\in\ZZ_M$. We will use $\calm(\ZZ_M)$ to denote the 
family of all such weighted multisets in $\ZZ_M$, and reserve the notation $A\subset \ZZ_M$ for sets (with 0-1 weights). 
If $A\in \calm(\ZZ_M)$, the polynomial $A(X)$ is sometimes called the {\it mask polynomial} of $A$.
It will usually be clear from the context whether $A$ refers to the weighted multiset or the corresponding polynomial; whenever there is any possibility of confusion, we will use $A$ for the multiset and $A(X)$ for the
polynomial. 

If $N|M$, then any $A\in \calm(\ZZ_M)$ induces a weighted multiset $A$ mod $N$ in $\ZZ_N$, with the corresponding mask polynomial $A(X)$ mod $(X^N-1)$, and induced weights 
\begin{equation}\label{induced-weights}
w_A^N(x)= \sum_{x'\in\ZZ_M: x'\equiv x\,{\rm mod}\, N} w_A(x'),\ \ x\in\ZZ_N.
\end{equation}
 For brevity, we will continue to write $A$ and $A(X)$ for $A$ mod $N$ and $A(X)$ mod $(X^N-1)$, respectively, while working in $\ZZ_N$.

 If $A,B\in\calm(\ZZ_M)$, we will use $A+B$ to indicate the weighted multiset corresponding to the sum of mask polynomials, or, equivalently, the sum of weight functions: 
 $$
 w_{A+B}(x)=w_A(x)+w_B(x),\ \ (A+B)(X)=A(X)+B(X).
 $$
 We will use the convolution notation $A*B$ to denote the weighted sumset of $A$ and $B$, so that $(A*B)(X)=A(X)B(X)$
 and
 $$w_{A*B}(x)=(w_A*w_B)(x)=\sum_{y\in\ZZ_M} w_A(x-y)w_B(y).$$
 If one of the sets is a singleton, say $A=\{x\}$, we will simplify the notation and write $x*B=\{x\}*B$.  
 The direct sum notation $A\oplus B$ is reserved for tilings, i.e., $A\oplus B=\ZZ_M$ means that $A,B\subset\ZZ_M$ are both sets and $A(X)B(X)=\frac{X^M-1}{X-1}$ mod $(X^M-1)$. 
 
 Since we will not need to use derivatives of polynomials in this article, we will use notation such as $A'$, $A''$, etc., to denote multisets and polynomials that need not have anything to do with the derivatives $\frac{d}{dX}A(X)$,  $\frac{d^2}{dX^2}A(X)$, and so on.


\subsection{Array coordinates}
\label{sec-array}


Suppose that $M=\prod_{i=1}^K p_i^{n_i}$,
where $p_1,\dots,p_K$ are distinct primes and $n_i\in\NN$. 
By the Chinese Remainder Theorem, we have
$$
\ZZ_M=\bigoplus_{i=1}^K \ZZ_{p_i^{n_i}}\, ,
$$
which we represent geometrically as a $K$-dimensional lattice. The tiling $A\oplus B=\ZZ_M$ can then be interpreted as a tiling of that lattice. 

It will be useful to have an explicit coordinate system on $\ZZ_M$. We fix one as follows.
Let $M_i = M/p_i^{n_i} = \prod_{j\neq i} p_j^{n_j}$, then each $x\in \ZZ_M$ can be written uniquely as
$$
x=\sum_{i=1}^K \pi_i(x) M_i,\ \ \pi_i(x)\in \ZZ_{p_i^{n_i}}.
$$
The mapping $x\to (\pi_1(x),\dots,\pi_k(x))$ identifies $x$ with an element of $\ZZ_{p_1^{n_1}}\times\dots 
\times \ZZ_{p_K^{n_K}}$.
We will refer to the $K$-tuple $(\pi_1(x),\dots,\pi_K(x))$ as the {\em $M$-array coordinates} of $x$.

We state a few easy properties for future reference.

\begin{lemma}\label{coord-properties}
(i) $x\equiv \pi_i(x)M_i \mod p_i^{n_i}$,

\smallskip
(ii) $(x,M)=\prod_{i=1}^K p_i^{\gamma_i}$ if and only if $p_i^{\gamma_i} \parallel \pi_i(x)$ for each $i=1,\dots,K$,

\smallskip
(iii) in particular, $x =0$ in $\ZZ_M$ if and only if $\pi_i(x)=0$ for each $i=1,\dots,K$,

\smallskip
(iv) if $x=\sum \pi_i(x)M_i$, $y=\sum \pi_i(y)M_i$, and $x+y = z =\sum \pi_i(z)M_i$ are the respective coordinate 
representations, then $\pi_i(z)\equiv \pi_i(x)+\pi_i(y) \mod p_i^{n_i}$  for each $i=1,\dots,k$.

\end{lemma}

Each coordinate $\pi_i(x)$ of $x\in \ZZ_M$ can be subdivided further into digits as follows. With
$\ZZ_{p_i^{n_i}}$ represented as $\{0,1,\dots, p_i^{n_i} -1\}$ with addition and multiplication modulo $p_i^{n_i}$,
we can write uniquely
$$
\pi_i(x) = \sum_{j=0}^{n_i-1} \pi_{i,j}(x) p_i^{j},\ \ \pi_{i,j}(x) \in\{0,1,\dots,p_i-1\}.
$$
Observe that 
$(x -x' ,M)=\prod_{j=1}^K p_i^{\gamma_i}$ if and only if for each $i=1,\dots,K$,
\begin{equation}\label{array-gcd}
\gamma_i = \begin{cases} 
\min\{j: \pi_{i,j}(x)\neq \pi_{i,j}(x') \} &\hbox{ if }  \pi_i(x)\neq \pi_i(x')
\\
n_i &\hbox{ if } \pi_i(x)=\pi_i(x').
\end{cases}
\end{equation}


\subsection{Grids, planes, lines, fibers}\label{gridsetal}


\begin{definition}\label{def-grid}
Let $D|M$. A $D$-grid in $\ZZ_M$ is a set of the form
$$
\Lambda(x,D):= x+D\ZZ_M=\{x'\in\ZZ_M:\ D|(x-x')\}
$$
for some $x\in\ZZ_M$.
\end{definition}

An important case of interest is as follows. Let $N|M$.
If $N=p_1^{\alpha_1} \dots p_K^{\alpha_K}$, with  $\alpha_1,\dots,\alpha_K\geq 0$, we let 
$$
D(N):= p_1^{\gamma_1} \dots p_K^{\gamma_K},
$$
where $\gamma_i=\max(0,\alpha_i-1)$ for $i=1,\dots, K$. Then a $D(N)$-grid is a ``top-level" grid on the scale $N$, and a natural setting to work on that scale.

While a grid $\Lambda$ is always an arithmetic progression in $\ZZ_M$, it is often helpful to represent $\ZZ_M$ by a $K$-dimensional coordinate array as in Section \ref{sec-array} and, accordingly, assign a geometric interpretation to $\Lambda$. We point out several useful special cases below.

A {\em line} through $x\in\ZZ_M$ in the $p_i$ direction is the set
$$
\ell_i(x):= \Lambda(x,M_i),
$$
and a {\em $(K-1)$-dimensional plane} through $x\in\ZZ_M$ perpendicular to the $p_i$ direction is a set of the form
\begin{equation}\label{def-plane}
\Pi(x,p_i^{\alpha_i}):=\Lambda(x,p_i^{\alpha_i}).
\end{equation}
Note that (\ref{def-plane}) defines a plane on the scale $M_ip_i^{\alpha_i}$, which may be different from $M$. 

An {\em $M$-fiber in the $p_i$ direction} is a set of the form $x*F_i$, where $x\in\ZZ_M$ and
\begin{equation}\label{def-Fi}
F_i=\{0,M/p_i,2M/p_i,\dots,(p_i-1)M/p_i\}.
\end{equation}
Thus $x*F_i=\Lambda(x,M/p_i)$. (More complicated multiscale fiber chains will be defined later.) A set $A\subset \ZZ_M$ is {\em $M$-fibered in the $p_i$ direction} if there is a subset $A'\subset A$ such that $A=A'*F_i$.


\subsection{Cyclotomic polynomials and cyclotomic divisibility}
\label{sec-cyclo}


We state a few basic facts about cyclotomic polynomials for future reference. By (\ref{poly-e0}), we have
\begin{equation}\label{polysec-e0}
1+X+X^2+\dots+X^{n-1}=\prod_{s|n,s\neq 1}\Phi_s(X).
\end{equation}
In particular, if $p$ is a prime number, then $\Phi_p(X)=1+X+\dots+X^{p-1}$, and, more generally, by induction
$$
\Phi_{p^\alpha}(X)=\Phi_p(X^{p^{\alpha-1}}) = 1 + X^{p^{\alpha-1}} + X^{2p^{\alpha-1}} + \dots + X^{(p-1)p^{\alpha-1}},
\ \ \alpha\geq 1.
$$
Thus $\Phi_{p^\alpha}(1)=p$, and this together with (\ref{polysec-e0}) implies that $\Phi_s(1)=1$ for all $s$ that are not prime powers. 

Suppose that $A\oplus B=\ZZ_M$, with $M=\prod_{i=1}^K p_i^{n_i}$ as before. By (\ref{poly-e1}), we have 
$A(X)B(X)=1+X+\dots+X^{M-1}$ mod $(X^M-1)$. For every prime power $s=p^\alpha|M$, we must have $\Phi_s(X)|A(X)B(X)$, so that 
$$
M= \prod_{i=1}^K \prod_{\alpha_i=1}^{n_i} \Phi_{p_i^{\alpha_i}}(1) \, \Big| \, A(1)B(1)=|A|\,|B|=M.
$$
It follows that 
$$
|A|= \prod_{s\in S_A} \Phi_s(1)
$$
and similarly for $B$, with $S_A$, $S_B$ defined as in Theorem \ref{CM-thm}; this is the proof of the tiling condition (T1) given in \cite{CM}. Moreover, for any prime power 
$s=p^{\alpha_i}|M$, we have that $\Phi_s(X)$ divides exactly one of $A(X)$ and $B(X)$. (This is not true for $s|M$ with two or more distinct prime factors. For such $s$, the corresponding cyclotomic polynomial $\Phi_s(X)$ may divide either one or both of $A(X)$ and $B(X)$.)

Divisibility by prime power cyclotomics has the following combinatorial interpretation. For $A\subset\ZZ_M$, the condition $\Phi_{p_i}|A$ means that the elements of $A$ are uniformly distributed modulo $p_i$, so that 
$$
|A\cap\Pi(x,p_i)|=|A|/p_i \ \ \forall x\in\ZZ_M.
$$
More generally, for $1\leq \alpha\leq n_i$, we have $\Phi_{p_i^\alpha}(X)|A(X)$ if and only if 
\begin{equation}\label{cyclo-uniform}
|A\cap \Pi(x,p_i^\alpha)|=\frac{1}{p_i} |A\cap \Pi(x,p_i^{\alpha-1})| \ \ \forall x\in\ZZ_M,
\end{equation} 
so that the elements of $A$ are uniformly distributed mod $p_i^\alpha$ within each residue class mod $p_i^{\alpha-1}$.
In particular, this implies the following bound on the number of points of a tile in a plane on a scale $M_ip_i^{n_i-\alpha_i}$, or, equivalently, in an arithmetic progression of step $p_i^{n_i-\alpha_i}$.

\begin{lemma}[\bf Plane bound]\label{planebound}
Let $A\oplus B=\ZZ_M$, where $M=\prod_j p_j^{n_j}$ and $|A|=\prod_j p_j^{\beta_j}$.
Then for every $x\in\ZZ_M$ and $0\leq\alpha_i\leq n_i$ we have 
\begin{equation}\label{e-planebound}
|A\cap\Pi(x,p_i^{n_i-\alpha_i})|\leq p_i^{\alpha_i}\prod_{\nu\neq i}p_\nu^{\beta_\nu}.
\end{equation}
\end{lemma}

The bound in Lemma \ref{planebound} is, in general, sharp. For example, if $A\subset p_i^{n_i-1}\ZZ_M$ and $\Phi_{p_i^{n_i}}|A$, then $|A\cap\Pi(x,p_i^{n_i})|=|A|/p_i$, so that (\ref{e-planebound}) holds with equality for $\alpha_i=0$. Examples of sets $A\subset\ZZ_M$ with the above properties are easy to construct using the standard tiling sets defined in Section \ref{standards}.

We also note the following. Let $N|M$. Then the condition
$$
\Phi_s|A \ \ \forall s|N, s\neq 1,
$$
means that $1+X+\dots+X^{N-1}$ divides $A(X)$, or, equivalently, that the elements of $A$ are uniformly distributed mod $N$. 
For example, let $N= p_1p_2\dots p_K$. Suppose that $|A|=N$ and that $\Phi_{p_i}|A$ for all $i=1,\dots,K$. Then $A$ satisfies (T2) if and only if $\Phi_s|A$ for all $s|N$ with $s\neq 1$, or, equivalently, if and only if each residue class mod $N$ contains exactly one element of $A$.


\subsection{Divisor set and divisor exclusion}


\begin{definition}[{\bf Divisor set}]\label{div-set}
For a set $A\subset\ZZ_M$, define
\begin{equation}\label{divisors}
\Div(A)=\Div_M(A):=\{(a-a',M):\ a,a'\in A\}.
\end{equation}
Informally, we will refer to the elements of $\Div(A)$ as the {\em divisors of $A$} or {\em differences in $A$}. We also define
$$
\Div_N(A):=\{(a-a',N):\ a,a'\in A\}
$$
for $A\subset\ZZ_M$ and $N|M$. 
\end{definition}

Divisor sets will be a key concept in our analysis, thanks to the following theorem due to Sands \cite{Sands}. 

\begin{theorem}\label{thm-sands} {\bf (Divisor exclusion; Sands \cite{Sands})}
If $A,B\subset \ZZ_M$ are sets, then $A\oplus B=\ZZ_M$
if and only if $|A|\,|B|=M$ and 
\begin{equation}\label{div-exclusion}
\Div(A) \cap \Div(B)=\{M\}.
\end{equation}
\end{theorem}

An alternative proof of Sands's Theorem, based on Theorem \ref{main-identity} due to \cite{GLW}, is included in 
Remark \ref{sands-proof}.


\section{A reformulation of T2}
\label{standards}


\subsection{Standard tiling complements}

We continue to assume that $M=\prod_{i=1}^K p_i^{n_i}$,
where $p_1,\dots,p_K$ are distinct primes and $n_i>0$. We equip $\ZZ_M$ with the array coordinate system from
Section \ref{sec-array} and use the notation of that section. Recall also that the divisor set
$\Div(A)$ for a set $A\subset\ZZ_M$ was defined in (\ref{divisors}).

\begin{definition}\label{standard}
Let $A,B$ be sets in $\ZZ_M$ such that $A\oplus B=\ZZ_M$. Let
\begin{equation*}
\mathfrak{A}_i(A) =\left\{\alpha_i \in \{1,2,\dots,n_i\}:\ \Phi_{p_i^{\alpha_i}}(X) | A(X) \right\}  
\end{equation*}
The {\em standard tiling set} $A^\flat$ is defined via its mask polynomial
\begin{equation}\label{replacement-factors}
\begin{split}
A^\flat(X) &= \prod_{i=1}^{K} \prod_{\alpha_i\in \mathfrak{A}_i(A)}  \Phi_{p_i}\left(X^{M_ip_i^{\alpha_i-1}}\right).
\\
&=  \prod_{i=1}^{K} \prod_{\alpha_i\in \mathfrak{A}_i(A)}  
\left(1+ X^{M_ip_i^{\alpha_i-1}} + \dots + X^{(p_i-1)M_ip_i^{\alpha_i-1}} \right).
\end{split}
\end{equation} 
\end{definition}

\begin{figure}[h]
\captionsetup{justification=centering}
\includegraphics[scale=1.5]{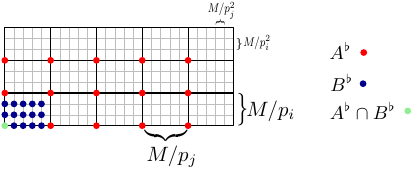}
\caption{The standard sets $A^\flat,B^\flat\subset\ZZ_{p_i^2p_j^2}$ with $p_i=3, p_j=5$\\
	 and $\Phi_{p_i^2}\Phi_{p_j^2}|A,\Phi_{p_i}\Phi_{p_j}|B$}
\end{figure}

\begin{lemma}\label{lemma-standard}
Define $A^\flat$ as above. Then $A^\flat(X)$ satisfies (T2) and has the same prime power cyclotomic divisors as $A(X)$. 
\end{lemma}

\begin{proof}
Let $\alpha\in\{1,2,\dots,n_i\}$. 
Then $\Phi_{p_i^{\alpha}}(X) | A^\flat(X)$ if and only if it divides one of the factors $\Phi_{p_i}(X^{M_ip_i^{\alpha_i-1}})$
with $\alpha_i \in \mathfrak{A}_i(A)$.
By Lemma \ref{cyclo-divisors} below, this happens if and only if $\alpha=\alpha_i$. 
Furthermore, in that case we also have
$\Phi_{d p_i^{\alpha_i}}(X) | A^\flat(X)$ for all $d|M_i$, so that in particular (T2) holds for $A^\flat$.
\end{proof}

\begin{lemma}\label{cyclo-divisors}
Let 
$$\Psi(X)=1+X^{Np^{\alpha-1}} + X^{2Np^{\alpha-1}} + \dots + X^{(p-1)Np^{\alpha-1}}
= \Phi_{p}\left(X^{Np^{\alpha-1}}\right),
$$
where $(N,p)=1$. Then $\Phi_s(X)|\Psi(X)$ if and only if $s=dp^\alpha$ for some $d|N$.
\end{lemma}

\begin{proof}

We have $\Phi_s(X)|\Psi(X)$ if and only if $\Psi(e^{2\pi i /s})=0$, i.e. $(e^{2\pi i /s})^{Np^{\alpha-1}}$
is a root of $\Phi_p$. This happens if and only if $(Np^{\alpha-1})/s = k/p$ for some integer $k$ such that
$(k,p)=1$. Equivalently, $Np^\alpha =ks$ with $(k,p)=1$. This means that $k|N$ and $s=\frac{N}{k}p^\alpha
= d p^\alpha$, where $d=N/k$ is a divisor of $N$.
\end{proof}

Observe that the standard set $A^\flat$, while not necessarily a grid, is highly structured. In terms of array coordinates, we have
\begin{equation}\label{flats}
\begin{split}
A^\flat &=  
\left\{ x\in\ZZ_M:\    \pi_i(x)=\sum_{\alpha_i\in \mathfrak{A}_i(A)} \pi_{i,\alpha_i-1}(x)p_i^{\alpha_i-1}, \   
\pi_{i,\alpha_i-1}(x) \in\{ 0,1,\dots, p_i-1   \}   \right\}
\\
&= \left\{x\in \ZZ_M:\ \pi_{i,\alpha_i-1}(x) =0 \hbox{ for all }i,\alpha_i \hbox{ such that }\alpha_i\notin \mathfrak{A}_i(A)\right\}.
\end{split}
\end{equation}

The {\em standard divisor set} for $A$ is
\begin{equation}\label{flat-divisors}
\Div (A^\flat ) = \left\{ \prod_{i=1}^K p_i^{\alpha_i-1}:\ \ \alpha_i\in \mathfrak{A}_i(A)\cup\{n_i+1\},\ i=1,\dots,K \right\}.
\end{equation}
We will refer to the elements of $\Div(A^\flat)$ as {\em standard divisors} of $A$. 
The set $B^\flat$ is defined similarly. 

With these definitions, we have the following alternative formulations of (T2).

\begin{proposition}\label{replacement} Suppose that $A\oplus B=\ZZ_M$. Then the following are equivalent:

(i) $\Div(A^\flat) \cap \Div(B)= \{M\}$,

(ii) $A^\flat \oplus B= \ZZ_M$,

(iii) $B$ satisfies (T2),

(iv) $|B\cap(x*A^\flat)|=1$ for every $x\in\ZZ_M$.

\end{proposition}

\begin{proof}
The equivalence between (i) and (ii) is a special case of Theorem \ref{thm-sands}. The implication (iii) $\Rightarrow$ (ii) follows from the construction in the proof of Theorem A in \cite{CM}; the converse implication (ii) $\Rightarrow$ (iii) was not pointed out there, but it also follows from the same construction.
Specifically, 
by Lemma \ref{cyclo-divisors}, $A^\flat(X)$ is divisible by every cyclotomic polynomial $\Phi_s$ such that
$p_i^\alpha\parallel s$ for some $i\in\{1,\dots,K\}$ and $\alpha\geq 1$ such that $\Phi_{p_i^\alpha}|A$. In other words, the only $s$ such that
$s|M$ but $\Phi_s$ does not divide $A(X)$ are those with $s=\prod_{i=1}^k p_i^{\beta_i}$, where for each $i$
we have either $\beta_i=0$ or $\Phi_{p_i^{\beta_i}}(X)|B(X)$. Let $\mathcal{S}_B$ be the set of such $s$.

If $B$ satisfies (T2), then all $\Phi_s$ with $s\in \mathcal{S}_B$ divide $B(X)$, which implies (ii). Conversely, suppose that (ii) holds. Then each $\Phi_s$ with $s|M$ has to divide $A(X)B(X)$. 
By Lemma \ref{cyclo-divisors} again, if $s\in \mathcal{S}_B$, then $\Phi_s$ does not divide $A(X)$, so it must divide $B(X)$.
Therefore (T2) holds for $B$.

For part (iv), we shall prove that (ii) implies (iv) and (iv) implies (i). Suppose that (ii) holds. 
We first claim that
\begin{equation}\label{add4}
|B\cap(x*A^\flat)|\leq 1\ \ \forall x\in\ZZ_M.
\end{equation}
Indeed, if $b,b'\in B\cap(x*A^\flat)$, then $b=x+a$ and $b'=x+a'$ for some $a,a'\in A^\flat$, so that $b-a=b'-a'$, contradicting (ii) unless $b=b'$ and $a=a'$. 

It remains to prove that $(x_0*A^\flat)\cap B\neq\emptyset$ for each $x_0\in\ZZ_M$. Let $x_0\in\ZZ_M$.
Since $\hbox{Div}(B)=\hbox{Div}(-B)$, we have $A^\flat\oplus (-B)=\ZZ_M$ by Theorem \ref{thm-sands}. 
It follows that there must exist $a_0\in A^\flat$ and $-b_0\in (-B)$ such that $a_0-b_0=-x_0$. The latter means $a_0+x_0=b_0$, implying $(x_0*A^\flat)\cap B\neq\emptyset$ as claimed. Hence (iv) follows.

%
%
%

Finally, suppose that (i) fails. Then there exist $b,b'\in B$ and $m\in\Div(A^\flat)\setminus\{M\}$ such that $(b-b',M)=m$. But then $b,b'\in B\cap(b*A^\flat)$ with $b\neq b'$, contradicting (iv).
\end{proof}

\begin{remark}\label{1dim-standard}
If $A\oplus B$ tiles $\ZZ_M$, where $M=p^n$ is a prime power, then (T2) holds vacuously for both sets. Hence we have both $A\oplus B^\flat=\ZZ_M$ and $A^\flat\oplus B=\ZZ_M$.
\end{remark}

It is tempting to try to prove that 
if $A\oplus B=\ZZ_M$, then we should have $\Div(A^\flat)\subseteq \Div(A)$. By Proposition \ref{replacement},
this would imply that $B$ satisfies (T2). However, 
the example below shows that $\Div(A)$ does not in fact have to contain $\Div(A^\flat)$.

\medskip\noindent
{\bf Example.} Let $M=p^2q$, where $p,q$ are distinct primes with $p>q$, and let $A$ be the set of numbers whose array coordinates in $\ZZ_M$ are
$(i+jp,i)$, $i=0,1,\dots,q-1$, $j=0,1,\dots,p-1$. Then $A\oplus B=\ZZ_M$ with $B=\{(j,0):\ j=0,1,\dots,p-1\}$,
and both sets satisfy T2.
Since $\Phi_{p^2}$ and $\Phi_q$ divide $A(X)$, we have $A^\flat= \{(jp,i): \ i=0,1,\dots,q-1, j=0,1,\dots,p-1\}$,
and in particular $p\in\Div(A^\flat)$. However, $p\notin \Div(A)$. To see this, consider $a,a'\in A$
with coordinates $(i+jp,i)$ and $(i'+j'p,i')$. If $i\neq i'$, then $(a-a',M)=1$. If on the other hand $i=i'$ but $a\neq a'$, then $(a-a',M)=pq$.

This shows that the condition $\Div(A^\flat)\subseteq \Div(A)$ is not necessary for a tiling, nor is it simply a 
consequence of $A(X)$ having the requisite prime power
cyclotomic divisors. Note, however, that in this example we still have $p\notin\Div(B)$.


\subsection{T2-equivalence}

\begin{definition}\label{t2-equi}
We say that the tilings $A\oplus B=\ZZ_M$ and $A'\oplus B=\ZZ_M$ are {\em T2-equivalent} if 
\begin{equation}\label{t2-equi-a}
A\hbox{ satisfies (T2) }\Leftrightarrow A'\hbox{ satisfies (T2).}
\end{equation}
\end{definition}

Since the sets $A$ and $A'$ tile the same group $\ZZ_M$ with the same tiling complement $B$, they must have the same cardinality and the same prime power cyclotomic divisors, as discussed in Section \ref{sec-cyclo}. For brevity, we will sometimes say simply that $A$ is T2-equivalent to $A'$ if both $M$ and $B$ are clear from context.

In practice, $A'$ will be a set obtained from $A$ using certain permitted manipulations such as fiber shifts. 
We will use T2-equivalence to reduce proving (T2) for the initial tiling to proving (T2) for related tilings that are increasingly more structured. In particular, the following reduction is sufficient to prove (T2) for both sets
in the given tiling.

\begin{corollary}\label{get-standard}
Suppose that the tiling $A\oplus B=\ZZ_M$ is T2-equivalent to the tiling $A^\flat \oplus B=\ZZ_M$. Then 
$A$ and $B$ satisfy (T2).
\end{corollary}

\begin{proof}
Since $A^\flat$ satisfies (T2), by (\ref{t2-equi-a}) so does $A$. By Proposition \ref{replacement}, $B$ satisfies
(T2) as well.
\end{proof}


\section{Box product}\label{sec-box-product}


\subsection{Box product characterization of tiling}

We continue to assume that $M=\prod_{i=1}^K p_i^{n_i}$, where $p_1,\dots,p_K$ are distinct primes. We
will use $\phi$ and $\mu$ to denote, respectively, the Euler totient function and the M\"obius function: if 
$n=\prod_{j=1}^L q_j^{r_j}$, where $q_1,\dots,q_L$ are distinct primes, then
$$
\phi(n)=n \prod_{j=1}^L \frac{q_j-1}{q_j} = \prod_{j=1}^L (q_j-1)q_j^{r_j-1},
$$
$$\mu(n)=\begin{cases} (-1)^L & \hbox{ if }r_1=r_2=\dots=r_L=1,
\\
0 & \hbox{  if } \exists j\in\{1,\dots,L\} \hbox{ such that }r_j\geq 2.
\end{cases}
$$

Let $N|M$. Reordering the primes if necessary, we
may assume that $N=p_1^{\alpha_1} \dots p_k^{\alpha_k}$, 
with $1\leq k\leq K$ and $\alpha_1,\dots,\alpha_k\geq 1$.

\begin{definition}\label{N-box} {\bf ($N$-boxes)}
An {\em $N$-box} is a $k$-dimensional matrix 
$$
\bbA=\big( \bbA_{(\gamma_1,\dots,\gamma_k)} \big) _{0\leq \gamma_j\leq \alpha_j,\ j=1,\dots,k}
$$
of size $(\alpha_1+1)\times \dots \times (\alpha_k+1)$,
with entries $\bbA_{(\gamma_1,\dots,\gamma_k)}\in\RR$. Since each multiindex $(\gamma_1,\dots,\gamma_k)$ with
$0\leq \gamma_j\leq \alpha_j$ can be uniquely associated with a divisor $m$ of $N$ given by
$m=p_1^{\gamma_1}\dots p_k^{\gamma_k}$, we will use such divisors to index the entries of $\bbA$, so that
$$
\bbA = ( \bbA_m )_{m|N}, \ \ \bbA_m = \bbA_{(\gamma_1,\dots,\gamma_k)}
\hbox{ for } m=p_1^{\gamma_1}\dots p_k^{\gamma_k}.
$$
\end{definition}

For any $N|M$, $N$-boxes form a vector space over $\RR$, with addition of boxes and 
multiplication of a box by a scalar defined in the obvious way. 
We also equip this space with an inner product structure as follows.

\begin{definition}\label{def-inner-product} {\bf (Box product)}
If $\bbA$ and $\bbB$ are $N$-boxes, define
\begin{equation}\label{inner-product}
\langle \bbA, \bbB \rangle = \sum_{m|N} \frac{1}{\phi(N/m)} \bbA_m \bbB_m.
\end{equation}
\end{definition}

Of course, (\ref{inner-product}) depends on $N$, but since $N$ is 
determined by the fact of $\bbA$ and $\bbB$ being $N$-boxes, we will not use additional subscripts or superscripts to indicate that.

The $N$-boxes associated with multisets in $\ZZ_N$ are as follows. 

\begin{definition}\label{N-box-A} {\bf (Boxes associated with multisets)}
Let $A\in\calm(\ZZ_M)$ and $N|M$. 
Consider the induced multiset $A\in\calm(\ZZ_N)$, with the weight function mod $N$ defined in (\ref{induced-weights}).
For $x\in\ZZ_N$, define
$\bbA^N[x] = (\bbA^N_m[x])_{m|N}$, 
where
\begin{align*}
\bbA^N_m[x] & = \sum_{a\in \ZZ_N:\, (x-a,N)=m} w^N_A(a).
\end{align*}
In particular, if $A\subset\ZZ_N$ is a set, we have
\begin{align*}
\bbA^N_m[x] & = \# \{a\in A:\ (x-a,N)=m \}.
\end{align*}
If $N=M$, we will skip the superscript and write $\bbA_m[x]=\bbA^M_m[x]$ whenever there is no possibility of confusion.
\end{definition}

Theorem \ref{ortho-lemma} below explains the reason for Definition \ref{def-inner-product}. 
The theorem is based on \cite[Theorem 1]{GLW} (see Sections \ref{GLW-section} and \ref{proofoftheorem3}
for details). 
The equivalence between $A\oplus B=\ZZ_M$ and the condition in (ii) provides an alternative proof of 
Theorem \ref{thm-sands};
however, Sands's proof is
easier and does not require Theorem \ref{ortho-lemma}.
The point of Theorem \ref{ortho-lemma} is that
tiling also implies the formally stronger condition (\ref{e-ortho}) for all $N|M$ and $x,y\in\ZZ_M$.

\begin{theorem}\label{ortho-lemma} {\bf (Box product characterization of tiling)}
(i) Suppose that $A\oplus B=\ZZ_M$ is a tiling. Then for any $N|M$, and for any $x,y\in\ZZ_M$, we have
\begin{equation}\label{e-ortho}
\langle \bbA^N[x], \bbB^N[y] \rangle =\frac{|A||B|}{N}= \frac{M}{N}.
\end{equation}
In particular, 
\begin{equation}\label{e-ortho2}
\langle \bbA^M[x], \bbB^M[y] \rangle =1\ \ \forall x,y\in\ZZ_M.
\end{equation}

(ii) Conversely, suppose that $A,B\subset \ZZ_M$ are sets such that $|A||B|=M$ and
$\langle \bbA^M[a], \bbB^M[b] \rangle =1$
for all $a\in A$ and $b\in B$. Then $A\oplus B=\ZZ_M$.
\end{theorem}

\begin{corollary}\label{cor-ortho}
Under the assumptions of Theorem \ref{ortho-lemma},
let $\call^N(A)$ be the linear space spanned by the boxes $\bbA^N[x]$, i.e.
$$
\call^N(A)=\left\{ \sum_{x\in\ZZ_N} c_x \bbA^N[x]:\ c_x\in\RR\right\},
$$
and similarly for $B$. Then for any $N$-boxes $\bbA\in\call^N(A)$ and $\bbB\in\call^N(B)$,
we have 
\begin{equation}\label{e-ortho2b}
\langle \bbA, \bbB \rangle = \frac{1}{N}\Sigma(\bbA) \Sigma(\bbB),
\end{equation}
where $\Sigma(\bbA)=\sum_{m|N} \bbA_m$.
\end{corollary}

$N$-boxes $\bbA^N[x]$, $x\in\ZZ_M$, are a convenient way of encoding structural information about $A$. Theorem \ref{ortho-lemma} provides a tiling criterion for $A\oplus B=\ZZ_M$ in terms of the box product, and it is also possible to express cyclotomic divisibility in terms of $N$-boxes. However, this convenience comes with some loss of information. In \cite{LaLo2}, we have to work with the actual sets $A$ and $B$, not just with the $N$-boxes representing them. We do not know whether it is possible to give a proof of properties such as (T2) purely in terms of the $N$-boxes associated with the sets.

\begin{remark} The equivalence between the conditions (iii) and (iv) in Proposition \ref{replacement} can be stated in terms of $M$-boxes. Suppose that $A\oplus B=\ZZ_M$ is a tiling. Then the following are equivalent:

(i) $B$ satisfies (T2),

(ii) $\sum_{m\in\Div(A^\flat)} \bbB_m[y]=1$ for all $y\in\ZZ_M$.

Indeed, the condition (iv) of Proposition \ref{replacement} can be rewritten as follows: for any $y\in\ZZ_M$, there exist unique $b\in B$ and $a\in A^\flat$ such that $b=y+a$, or equivalently, $y-b=-a$. Since $\{a\in\ZZ_M:\ (-a,M)\in\Div(A^\flat)\}=A^\flat$, this implies the claim.
\end{remark}


\subsection{A Fourier-analytic identity}\label{GLW-section}


Fix $N=p_1^{\alpha_1} \dots p_k^{\alpha_k}$, where $p_1,\dots,p_k$ are distinct primes and $\alpha_1,\dots,\alpha_k\in\NN$. 
Let $A,B,C,D\in\calm(\ZZ_N)$. For $m|N$, we define
$$
\bbA_m^N [C] :=\sum_{c\in C}\bbA^N_m[c] w_C(c) =\sum_{a,c\in\ZZ_N} w_A(a) w_C(c) \one_{(a-c,N)=m}.
$$
In particular, if $A(X)$ is a polynomial with 0-1 coefficients corresponding to a set $A\subset\ZZ_N$, then 
$$
\bbA_m^N[A]= \#\{(a,a')\in A\times A:\ (a-a',N)=m\}.
$$
This defines $N$-boxes in the sense of Definition \ref{N-box-A}, and in particular we may consider the box product
$$
\langle \bbA^N[C], \bbB^N[D] \rangle = \sum_{m|N} \frac{1}{\phi(N/m)}  \bbA_m^N[C] \bbB_m^N[D].
$$ 
 
The following theorem is a slight generalization of the main identity from \cite{GLW}. 
Specifically, Theorem 1 in \cite{GLW} states that (\ref{id-1}) holds when $A(X)$ and $B(X)$
are polynomials corresponding to multisets $A,B\subset\ZZ_N$. We will need an extension
of it to 4 polynomials, not necessarily with non-negative coefficients.
The proof is essentially the same, but since \cite{GLW} remains unpublished, we include it here for completeness.

\begin{theorem}\label{main-identity}
Let $A(X),B(X),C(X),D(X)$ be polynomials modulo $X^N-1$ with integer coefficients. Then
\begin{equation}\label{id-1var}
\langle \bbA^N[C], \bbB^N[D] \rangle=\sum_{d|N} \frac{1}{N\phi(d) } 
\left[ \sum_{\zeta:\Phi_d(\zeta)=0} A(\zeta)\overline{C(\zeta)} \right]
\left[ \sum_{\zeta:\Phi_d(\zeta)=0} B(\zeta)\overline{D(\zeta)} \right].
\end{equation}
In particular, 
\begin{equation}\label{id-1}
\langle \bbA^N[A], \bbB^N[B] \rangle
=\sum_{d|N} \frac{1}{N\phi(d) } \cale_d(A) \cale_d(B),
\end{equation}
where
$$
\cale_d(A)=\sum_{\zeta:\Phi_d(\zeta)=0} |A(\zeta)|^2.
$$
\end{theorem}

The rest of this section is devoted to the proof of Theorem \ref{main-identity}.
We will use the discrete Fourier transform in $\ZZ_N$:
if $f:\ZZ_N\to \CC$ is a function, then
$$
\widehat{f}(\xi)= \sum_{x\in\ZZ_N} f(x) e^{2\pi i x\xi/N},\ \ \xi\in\ZZ_N.
$$

\begin{lemma}
Let
$$
\Lambda_m  :=\{x\in\ZZ_N: \ m|x\},
$$
$$
H_m  := \{x\in\ZZ_N: \ (x,N)=m\} = \Lambda_m \setminus \bigcup_{m': \,m|m'|N,\,m'\neq m} \Lambda_{m'}.
$$
Then
$\widehat{\one_{H_m}}(\xi)=G_\xi(N/m)$, where
\begin{equation}\label{e-defG}
G_\xi(v)=\sum_{d|(v,\xi)} \mu(v/d) d.
\end{equation}
\end{lemma}

\begin{proof}
Using that 
$$
\widehat{\one_{\Lambda_m}} (\xi)= \frac{N}{m} \one_{\Lambda_{N/m}}(\xi),
$$
we get by inclusion-exclusion
\begin{align*}
\widehat{\one_{H_m}}(\xi)
&= \sum_{d|\frac{N}{m}} \mu(d)
\widehat{\one_{\Lambda_{md}}}(\xi) 
\\
&=  \sum_{d|\frac{N}{m}} \mu(d) \,\frac{N}{md} \,
\one_{\Lambda_{N/md}}(\xi) 
\\
&=  \sum_{d'|\frac{N}{m} } \mu\Big( \frac{N}{md'} \Big) d' \, \one_{\Lambda_{d}}(\xi)
\\
&=  \sum_{d'| (\frac{N}{m},\xi) } \mu\Big( \frac{N}{md'} \Big) d' 
\\
&= G_\xi(N/m).
\end{align*}
\end{proof}

\begin{proposition}\label{p-mainG}
We have
\begin{equation}\label{e-mainG}
\sum_{v|N} \frac{1}{\phi(v)} G_\xi(v) G_{\xi'}(v) =
\begin{cases}
\frac{N}{\phi(N/d)} & \hbox{ if } (\xi,N)=(\xi',N)=d
\\
0 & \hbox{  if } (\xi,N)\neq (\xi',N).
\end{cases}
\end{equation}
\end{proposition}

\begin{proof}
We first claim that for every $\xi$, the function $G_\xi(v)$ is multiplicative in $v$. Indeed, let $(x,y)=1$, $(x,\xi)=t$, $(y,\xi)=s$. Then $(t,s)=1$ and $(xy,\xi)=ts$. Writing $u=u^\prime\cdot u^{\prime\prime}$, $u^\prime|x$ and $u^{\prime\prime}|y$, we get 
$$
G_\xi(xy)=\sum_{u|(xy,\xi)} \mu(\frac{xy}{u}) u=\sum_{u^\prime|(x,\xi),u^{\prime\prime}|(y,\xi)} \mu(\frac{x}{u^\prime}) \mu(\frac{y}{u^{\prime\prime}}) u^\prime u^{\prime\prime}=G_\xi(x)G_\xi(y).
$$
Therefore $G_\xi(v)$ is entirely determined by its values $G_\xi(v)$ on prime powers $p_i^j$, $j=0,1,\dots,\alpha_i$, $i=1,\dots,k$. 
Let $\xi=p_1^{\nu_1}\dots p_k^{\nu_k}$.
Then $p_i^\kappa \parallel(p_i^j,\xi)$ for $\kappa=\min(j,\nu_i)$. If $j=0$, we have $\kappa=0$ and $G_\xi(p_i^0)=G_\xi(1)=1$.
If $j\geq 1$, we have
$$
G_\xi(p_i^j)=\sum_{u|p_i^\kappa}\mu\left(\frac{p_i^j}{u}\right) u=
\begin{cases}
p_i^j -p_i^{j-1} & \hbox{ if } \kappa=j\\
-p_i^{j-1} &   \hbox{ if }    \kappa=j-1\\
0 &   \hbox{ if }    \kappa <j
\end{cases}	
$$
which is equivalent to
\begin{equation}\label{G-e2}
G_\xi(p_i^j)=
\begin{cases}
p_i^j -p_i^{j-1}=\phi(p_i^j) & \hbox{ if } j\leq \nu_i\\
-p_i^{\nu_i} &   \hbox{ if }     j=\nu_i+1\\
0 &   \hbox{ if }      j>\nu_i+1
\end{cases}
\end{equation}
Next, if $F$ is a multiplicative function, then 
\begin{align*}
\sum_{v|N}F(v) 
=\sum_{0\leq j_1\leq \alpha_1,\, \ldots,\, 0\leq j_k\leq \alpha_k} F(p_{1}^{j_1}\ldots p_{k}^{j_k})
=\prod_{i=1}^k \left(\sum_{j_i=0}^{\alpha_i}F(p_i^{j_i})\right)
\end{align*}
Applying this with $F(v)=G_\xi(v)G_{\xi'}(v)$ for fixed $\xi,\xi'$, we get
\begin{equation}\label{G-e1}
\sum_{v|N} \frac{1}{\phi(v)} G_\xi (v) G_{\xi^\prime} (v)
=\prod_{i=1}^k \left( \sum_{j_i=0}^{\alpha_i} \frac{1}{\phi(p_i^{j_i})}G_\xi (p_i^{j_i})G_{\xi^\prime} (p_i^{j_i}) \right)
\end{equation}

Fix a prime divisor $p_i|N$, and consider the corresponding factor in (\ref{G-e1}).
Suppose that $0\leq \nu,\nu' \leq \alpha_i$ are such that $p_i^\nu\parallel \xi$, $p_i^{\nu'} \parallel \xi^\prime$. Without loss of generality, we may assume that $\nu\leq\nu'$. In accordance with
(\ref{G-e2}), we have three cases.
 \begin{itemize}
	\item If $\nu<\nu'\leq \alpha_i$, then $\nu+1\leq \alpha_i$ and
	\begin{align*}
	\sum_{j=0}^{\alpha_i} \frac{1}{\phi(p_i^j)} G_\xi(p_i^j) G_{\xi^\prime}(p_i^j) 
	& =1+\sum_{j=1}^{\nu} \frac{1}{\phi(p_i^j)} \phi(p_i^j) (p_i^j -p_i^{j-1})
	+\frac{1}{\phi(p_i^{\nu+1})} \phi(p_i^{\nu+1}) (-p_i^n)\\
	& = 1+\sum_{j=1}^{\nu} (p_i^j -p_i^{j-1}) -p^\nu=0
	\end{align*}	
	\item If $\nu=\nu' < \alpha_i$, then $\nu+1=\nu' +1 \leq \alpha_i$ and
	$$
	\sum_{j=0}^{\alpha_i} \frac{1}{\phi(p_i^j)} G_\xi(p_i^j) G_{\xi^\prime}(p_i^j) 
	=1+\sum_{j=1}^{\nu} (p_i^j-p_i^{j-1})+\frac{1}{p_i^{\nu+1}-p_i^\nu} (-p_i^\nu)^2=\frac{p_i^{\nu+1}}{p_i-1}
	$$
	\item If $\nu=\nu'=\alpha_i$, then
	$$
	\sum_{j=0}^{\alpha_i} \frac{1}{\phi(p_i^j)} G_\xi(p_i^j) G_{\xi^\prime}(p_i^j) 
	=1+\sum_{j=1}^{\alpha_i} (p_i^j-p_i^{j-1})=p_i^{\alpha_i}
	$$
\end{itemize}
Since
$$
\frac{p_i^{\alpha_i}}{\phi(p_i^{\alpha_i-\nu})}= 
\begin{cases}
p^{\alpha_i} & \hbox{ if }\nu=\alpha_i \\
\frac{p^{\alpha_i}}{(p_i-1)p_i^{\alpha_i-\nu-1}} = \frac{p_i^{\nu+1}}{p_i-1} & \hbox{ if } \nu<\alpha_i
\end{cases}
$$
we conclude that
$$
\sum_{j=0}^{\alpha_i} \frac{1}{\phi(p_i^j)} G_\xi(p_i^j) G_{\xi^\prime}(p_i^j)= \begin{cases}
0 & \hbox{ if }\nu\neq\nu'\\
\frac{p_i^{\alpha_i}}{\phi(p_i^{\alpha_i-\nu})}
& \hbox{ if } \nu=\nu'\\
\end{cases}
$$	
We now plug this into (\ref{G-e1}), and, since $p_i$ is no longer fixed, write $\nu_i$ and $\nu'_i$ instead of $\nu$ and $\nu'$. 
If $(\xi,N)\neq(\xi^\prime,N)$, then 
$\nu_i\neq \nu'_i$ for at least one $p_i$, so that  
$$
\sum_{v|N} \frac{1}{\phi(v)} G_\xi (v) G_{\xi^\prime} (v)=0.
$$
If on the other hand $(\xi,N)=(\xi^\prime,N)$, we get
$$
\sum_{v|N} \frac{1}{\phi(v)} G_\xi (v) G_{\xi^\prime} (v)
=\prod_{i=1}^k \frac{p_i^{\alpha_i} }{\phi(p_i^{\alpha_i-\nu_i})}
=\frac{N}{\phi(N/(\xi,N))},
$$
which proves the proposition.
\end{proof}

In order to finish the proof of Theorem \ref{main-identity}, we write
\begin{equation*}
\begin{split}
 \bbA^N_m[C]
 & =\sum_{x,y,z\in\ZZ_N} w_A(x)w_C(y)\mathbf{1}_{H_m}(z)\mathbf{1}_{x-y=z}
\\
& =\frac{1}{N}\sum_{x,y,z\in\ZZ_N} w_A(x)w_C(y)\mathbf{1}_{H_m}(z) \sum_{\xi \in \ZZ_N} e^{-2\pi i \xi (x-y+z)/N}
\\
& = \frac{1}{N} \sum_{\xi \in \ZZ_N}  \widehat{w_A}(\xi)   \overline{ \widehat{w_C}(\xi) }
\widehat{\mathbf{1}_{H_m}}(\xi)
\\
& = \frac{1}{N} \sum_{\xi \in \ZZ_N}  \widehat{w_A}(\xi)   \overline{ \widehat{w_C}(\xi) }
G_\xi (N/m)
\end{split}
\end{equation*}
Therefore
\begin{align*}
\langle \bbA^N[C], \bbB^N[D] \rangle
& =\frac{1}{N^2} \sum_{v|N} \frac{1}{\phi(v)} \left(\sum_{\xi \in \ZZ_N}
 \widehat{w_A}(\xi)   \overline{ \widehat{w_C}(\xi) }
G_\xi (v)\right)\left(\sum_{\xi^\prime \in \ZZ_N}
 \widehat{w_B}(\xi')   \overline{ \widehat{w_D}(\xi') }
 G_{\xi^\prime} (v)\right)
\\
& = \frac{1}{N^2} \sum_{\xi, \xi^\prime \in \ZZ_N}
 \widehat{w_A}(\xi)   \overline{ \widehat{w_C}(\xi) }
  \widehat{w_B}(\xi')   \overline{ \widehat{w_D}(\xi') }
   \left[\sum_{v|N} \frac{1}{\phi(v)} G_\xi (v) G_{\xi^\prime} (v)\right]
\end{align*}

By Proposition \ref{p-mainG},  
\begin{align*}
\langle \bbA^N[C], \bbB^N[D] \rangle
& =\frac{1}{N^2} \sum_{d|N} \frac{N}{\phi(N/d)} \left(\sum_{\xi: (\xi,N)=d}
 \widehat{w_A}(\xi)   \overline{ \widehat{w_C}(\xi) } \right)
 \left(\sum_{\xi^\prime: (\xi^\prime,N)=d}
   \widehat{w_B}(\xi')   \overline{ \widehat{w_D}(\xi') } \right),
\end{align*}
which is (\ref{id-1}), since $ \widehat{w_A}(\xi)  = A(e^{-2\pi i \xi/N})$ and 
$\zeta= e^{-2\pi i \xi/N}$ is a root of $\Phi_d(X)$ if and only if $(\xi,N)=N/d$.


\subsection{Proof of Theorem \ref{ortho-lemma}}\label{proofoftheorem3}

(i) Assume that $A\oplus B=\ZZ_M$. Let $N|M$, and let $C=\{x\}$ and $D=\{y\}$ for $x,y\in\ZZ_N$. By (\ref{id-1var}), we have
\begin{align*}
\langle \bbA^N[x], \bbB^N[y] \rangle 
=\sum_{d|N} \frac{1}{N \phi(d) } 
\left[ \sum_{\zeta:\Phi_d(\zeta)=0} A(\zeta)\overline{C(\zeta)} \right]
\left[ \sum_{\zeta:\Phi_d(\zeta)=0} B(\zeta)\overline{D(\zeta)} \right].
\\
\end{align*}
If $d\neq1$, then $\Phi_d(X)$ divides at least one of $A(X)$ and $B(X)$. Hence the only non-zero contribution is 
from $d=1$, which yields
$$
\langle \bbA^N[x], \bbB^N[y] \rangle = \frac{1}{N} A(1)C(1)B(1)D(1) = \frac{|A| |B|}{N}.
$$
This proves (\ref{e-ortho}).

(ii) Suppose that $A,B\subset\ZZ_M$ satisfy $\langle \bbA^M[a], \bbB^M[b] \rangle=1$
for all $a\in A,b\in B$. Then
\begin{align*}
\langle \bbA^M[A], \bbB^M[B] \rangle
&=  \sum_{a\in A,b\in B} \langle \bbA^M[a], \bbB^M[b] \rangle\\
&=  \sum_{a\in A,b\in B} 1  = |A| |B|=M.
\end{align*}
By (\ref{id-1}), this implies that 
$$
\sum_{d|M} \frac{1}{\phi(d) } \cale_d(A) \cale_d(B) = M^2.
$$
However, we are also assuming that $\cale_1(A) \cale_1(B)= |A|^2 |B|^2 = M^2$. Hence
$\cale_d(A) \cale_d(B)=0$ for all $d|M$, $d\neq 1$, so that $A\oplus B=\ZZ_M$ as claimed.

\medskip

\begin{remark}\label{sands-proof}
We indicate a proof of Theorem \ref{thm-sands} based on (\ref{id-1}). Suppose that $A\oplus B=\ZZ_M$ is a tiling, and apply (\ref{id-1}) with $N=M$. Since $\Phi_d(X)|A(X)B(X)$ for all $d|M$, $d\neq 1$, we get
$$
\sum_{m|M}\frac{1}{\phi(M/m)} \bbA^M_m[A]\bbB^M_m[B] =
\frac{|A(1)|^2 |B(1)|^2}{M}=M.
$$
But we also have $\bbA^M_M[A]\bbB^M_M[B]=|A||B|=M$. Hence all other terms $\bbA^M_m[A]\bbB^M_m[B]$ with $m\neq M$ must vanish. This proves Theorem \ref{thm-sands}.

\end{remark}


\section{Cuboids}
\label{gen-cuboids}



\subsection{Definitions}


\begin{definition}
\label{def-gen-cuboids}

Let $M=\prod_{i=1}^K p_i^{n_i}$. 

\smallskip

(i) A {\em cuboid type} $\calt$ on $\ZZ_N$ is an ordered triple $\calt=(N,\vec{\delta}, T)$, where:

\begin{itemize}
\item $N=\prod_{i=1}^K p_i^{n_i-\alpha_i}$ is a divisor of $M$, with $0\leq \alpha_i\leq n_i$ for each $i=1,\dots,K$,

\item $\vec{\delta}=(\delta_1,\dots,\delta_K)$, with $0\leq \delta_i\leq n_i-\alpha_i$ for $i=1,\dots,K$, 

\item $T\subset\ZZ_N$ is a nonempty set.

\end{itemize}

We will refer to $N$ as the {\em scale} of $\calt$, and to $T$ as its {\em template}.
We also define $\mathfrak{J} =\mathfrak{J}_{\vec\delta} := \{j: \delta_j\neq 0\}$.

\medskip

(ii) Let $\calt =(N,\vec{\delta}, T)$ be a cuboid type as above. A    
{\em cuboid} $\Delta$ of type $\calt$ is a weighted multiset corresponding to a mask polynomial of the form
\begin{equation}\label{gc2}
\Delta(X)= X^c\prod_{j\in \mathfrak{J}} (1-X^{d_j}),
\end{equation}
where $c,d_j$ are elements of $\ZZ_M$ such that $(d_j,N)=N/p_j^{\delta_j}$.
The {\em vertices} of $\Delta$ are the points
\begin{equation}\label{gc1}
x_{\vec\epsilon}=c+\sum_{j\in \mathfrak{J}} \epsilon_jd_j:\ \vec{\epsilon} = (\epsilon_j)_{j\in \mathfrak{J}} \in\{0,1\}^{|\mathfrak{J}|},
\end{equation}
with weights
$w_\Delta(x_{\vec\epsilon})=(-1)^{\sum_{j\in \mathfrak{J}}\epsilon_j}$.

\medskip  
(iii) 
Let $A\in\calm(\ZZ_N)$, and let $\Delta$ be a cuboid of type $\calt$. Then the {\em $(\Delta,T)$-evaluation of $A$} is
$$
\bbA^\calt [\Delta] = \bbA^N_N[\Delta*T]= \sum_{  \vec\epsilon\in\{0,1\}^k} w_\Delta(x_{\vec\epsilon}) \bbA^N_N[x_{\vec\epsilon}*T],
$$
where we recall that $x*T=\{x+t: \ t\in T\}$, so that
$$
\bbA^N_N[x_{\vec\epsilon}*T]:= \sum_{t\in T} \bbA^N_N[x_{\vec\epsilon}+t].$$
\end{definition}

For consistency, we will also write
$$
\bbA^\calt [x] = \bbA^N_N[x*T],\ \ x\in\ZZ_M.
$$
In some situations, it will be easier to write out $\Delta$ in its polynomial form. We will then identify the polynomial $\Delta(X)$ with the corresponding weighted multiset $\Delta$, and write $\bbA^\calt [\Delta(X)]$ instead of $\bbA^\calt [\Delta]$.

\begin{definition}
\label{def-null}
Let $A\in\calm(\ZZ_M)$, and let $\mathcal{T} = (N,\vec\delta,T)$ be a cuboid type
 as above. We will say that 
{\em $A$ is $\mathcal{T}$-null} if for every cuboid $\Delta$ of type $\mathcal{T}$, 
\begin{equation}\label{tnull}
\bbA^\calt [\Delta] =0.
\end{equation}
\end{definition}

\begin{lemma}\label{cub-div}
Let $A\in\calm(\ZZ_M)$, and let $\mathcal{T} = (N,\vec\delta,T)$ be a cuboid type. 
Suppose that
for all $m|N$, the cyclotomic polynomial $\Phi_m(X)$ divides at least one of $A(X)$, $T(X)$, or $1-X^{N/p_j^{\delta_j}}$ 
for some $j\in \mathfrak{J}(\vec\delta)$. Then $A$ is $\calt$-null.
\end{lemma}

\begin{proof}
This follows e.g. from 
Theorem \ref{main-identity} applied to $A(X)$ and $C(X)=\Delta(X)T(X)$, with $B=D=\{0\}$.
\end{proof}


\subsection{Classic cuboids}\label{classic-cuboids}


\begin{definition}
\label{def-N-cuboids}
An {\em $N$-cuboid} is a cuboid of type $\calt=(N,\vec{\delta},T)$, where 
$N|M$, $T(X)=1$, and $\delta_j=1$ for all $j$ such that $p_j|N$. Thus, $N$-cuboids have the form
$$\Delta(X)= X^c\prod_{p_j|N} (1-X^{\rho_jN/p_j})$$
 with $(\rho_j,p_j)=1$ for all $j$, and the associated $\Delta$-evaluation of a multiset $A\in\calm(\ZZ_N)$ is
$$
\bbA^N_N[\Delta] = \sum_{  \vec\epsilon\in\{0,1\}^{|\mathfrak{J}|}} w_\Delta(x_{\vec\epsilon}) \bbA^N_N[x_{\vec\epsilon}],
$$
where $\mathfrak{J}=\{j:\ p_j|N\}$ and the cuboid vertices $x_{\vec\epsilon}$ are defined in (\ref{gc1}). If $\calt$ is as above and $A\in\calm(\ZZ_N)$ is $\calt$-null, we will also say for short that $A$ is $N$-null.
\end{definition}

\begin{figure}[h]
\captionsetup{justification=centering}
\includegraphics[scale=1]{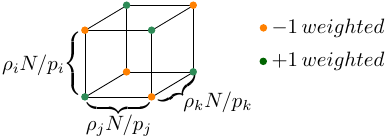}
\caption{An $N$-cuboid with $N$ having 3 prime factors}
\end{figure}

The geometric interpretation of $N$-cuboids $\Delta$ is as follows. With notation as in Definition \ref{def-N-cuboids}, 
recall that $D(N)=N/\prod_{j\in \mathfrak{J}}p_j$. Then the vertices 
$x_{\vec\epsilon}$ of $\Delta$ form a full-dimensional rectangular box in the grid $\Lambda(c,D(N))$, with one vertex at $c$ and alternating $\pm 1$ weights. We reserve the term ``$N$-cuboid", without cuboid type explicitly indicated, to refer to cuboids as in Definition \ref{def-N-cuboids}; for cuboids of any other type, we will always specify $\calt$.

The following cyclotomic divisibility test has been known and used previously in the literature,
see e.g.  \cite[Section 3]{Steinberger} in the context of vanishing sums of roots of unity, or \cite[Section 3]{KMSV} and \cite{KMSV2} with applications to the ``spectral implies tiling" direction of Fuglede's conjecture.

\begin{proposition}\label{cuboid}
Let $A\in\calm(\ZZ_N)$. Then the following are equivalent:

\medskip

(i) $\Phi_N(X)|A(X)$,

\medskip

(ii) For all $N$-cuboids $\Delta$, we have
\begin{equation}\label{id-3a}
\bbA^N_N[\Delta]=0.
\end{equation}

\end{proposition}

\begin{proof}
Let $m|N$ satisfy $m\neq N$. Then $m|(N/p_i)$ for some $i$ such that $p_i|N$, so that $\Phi_m|(1-X^{N/p_i})$.
The implication (i) $\Rightarrow$ (ii) now follows from Lemma \ref{cub-div}.

An alternative proof that (i) implies (ii) (without using Theorem \ref{main-identity}; 
cf. \cite{Steinberger}, \cite{KMSV}) is as follows.
By classic results on vanishing sums of roots of unity (see \cite{deB}, \cite{LL}, \cite{Mann}, \cite{Re1}, \cite{Re2}, \cite{schoen}), $\Phi_N(X)|A(X)$ if and only if $A(X)$ is a linear combination of
the polynomials $\Phi_p(X^{N/p})$, where $p$ runs over all prime divisors of $N$, with integer (but not necessarily nonnegative) coefficients.
Equivalently, $\Phi_N(X)|A(X)$ if and only if 
$A$ can be represented as a 
linear combination of $N$-fibers.
It is very easy to see that 
(\ref{id-3a}) holds for all $N$-cuboids $\Delta$ if $A$ mod $N$ is an $N$-fiber, therefore it
also holds if $A$ mod $N$ is a linear combination of such fibers.

The proof that (ii) implies (i) is by induction on the number of prime divisors of $N$ (this argument was also 
known previously in the literature, see e.g.
\cite[Proposition 2.4]{Steinberger}). We present it here for completeness.

If $N=p^\alpha$
is a prime power, the claim follows from (\ref{cyclo-uniform}).
Suppose that the claim is true for all $N'$ with at most $k$ prime
divisors. Suppose that $N$ has $k+1$ prime divisors, and that $A\in\calm(\ZZ_N)$ obeys
(\ref{id-3a}) for all $N$-cuboids $\Delta$ in $\ZZ_N$. Let $p$ be a prime divisor of $N$,
and let $N'=N/p^\alpha$, where $(N',p)=1$. 

Assume first that 
\begin{equation}\label{e-skip}
A\in\calm( p^{\alpha-1}\ZZ_N).
\end{equation}
Write $A(X)=\sum_{j=0}^{p-1} X^{jN/p } A_j(X)$, where 
$A_j\in\calm(p^{\alpha}\ZZ_N)$. 
Each ``layer" $A_j$ can be identified in the
obvious manner with a multiset in $\ZZ_{N'}$.

For $j=0,1,\dots,p-1$, let $A_{j,0}$ be the weighted multiset defined via ${A}_{j,0}(X)= {A}_j(X)-{A}_0(X)$. The condition (\ref{id-3a}) 
implies that, with the obvious notation,
$$
(\bbA_{j,0})^{N'} [\Delta'] =0
$$
for every full-dimensional cuboid $\Delta'$ in $\ZZ_{N'}$. By the inductive assumption, $\Phi_{N'}(X)|
A_{j,0}(X)$. By the structure theorem for vanishing sums of roots of unity, ${A}_{j,0}$ is a linear combination
of $N'$-fibers in $\ZZ_{N'}$. Returning to $\ZZ_N$ now, and summing in $j$, we get that $A(X)
= A'(X)+A''(X)$, where:
\begin{itemize}
\item 
$A'(X)=\sum_{j=0}^{p-1} X^{jN/p } A_{j,0}(X)$.
 By the above argument, ${A}'$ is a linear combination of fibers in directions
perpendicular to $p$. 

\item ${A}''=\sum_{j=0}^{p-1} X^{jN/p} A_0(X)$.
This is a linear combination of fibers in the $p$ direction.
\end{itemize}
Thus ${A}$ is a linear combination of fibers, and therefore $\Phi_N(X)|A(X)$.

Finally, in the general case (without assuming (\ref{e-skip}), we can write $A$ as a union of multisets
$A^{(i)}$, $i=0,1,\dots, p^{\alpha-1} -1$, where each $A^{(i)}$ is a translate of a multiset satisfying (\ref{e-skip}).
If (\ref{id-3a}) holds for $A$, then it also holds for each $A^{(i)}$. By the previous argument we get that
$\Phi_N(X)|A^{(i)}(X)$ for each $i$, therefore it divides $A(X)$. This completes the proof that (ii) implies (i).
\end{proof}

\begin{remark}
Proposition \ref{cuboid} implies in particular that, for any $N|M$, $\Phi_N$ divides $A$ if and only if it divides the mask polynomial of $A\cap\Lambda(x,D(N))$ for every $x\in\ZZ_M$. Indeed, the vertices of any $N$-cuboid $\Delta$ are all contained in the same $D(N)$-grid. Hence the divisibility of $A$ by $\Phi_M$ is associated with the structure of $A$ on such grids.
\end{remark}


\subsection{Multiscale cuboids}\label{multiscale-cuboids}


In many situations, we need to work with cuboids on several scales simultaneously. This happens, for example, when we investigate divisibility of a polynomial $A(X)$ by combinations of cyclotomic polynomials, or when we try to reduce a tiling of $\ZZ_M$ to tilings of cosets of a subgroup. 
We will use cuboids with nontrivial templates to facilitate such multiscale cuboid analysis.

\begin{definition}\label{def-folding}
{\bf (Folding templates)} Let $M=\prod_{i=1}^K p_i^{n_i}$ and
$N=\prod_{i=1}^K p_i^{n_i-\alpha_i}$, with $0\leq \alpha_i\leq n_i$ for $i=1,\dots,K$. The {\em folding template} 
$T^M_N$ is given by
$$
T^M_N(X)
= \prod_{i:p_i|\frac{M}{N}} \prod_{\nu_i=1}^{\alpha_i} \Psi_{M/p_i^{\nu_i}}(X)
\equiv \frac{X^M-1}{X^N-1} \mod (X^M-1)
$$
where 
\begin{equation}\label{sf-def}
\Psi_{M/p_i^\delta}(X)= \Phi_{p_i}(X^{M/p_i^\delta}) = 1+X^{M/p_i^{\delta}}+ X^{2M/p_i^{\delta}} + \dots + X^{(p_i-1) M/p_i^{\delta}}.
\end{equation}
When $M$ is fixed, we will sometimes write $T_N$ instead of $T^M_N$ for simplicity. 
\end{definition}

Strictly speaking, $\Psi_{M/p_i^\delta}$ depends on both $M$ and $p_i^\delta$, and not just on their quotient; however, both numbers will always be clear from the context. We also note that $\Psi_{M/p_i}=F_i$.

Definition \ref{def-folding} allow us to consider $N$-cuboids as cuboids with templates in $\ZZ_M$. Specifically, 
let $N|M$ be as in Definition \ref{def-folding}.
Then for any $A\in\calm(\ZZ_M)$ and $x\in\ZZ_M$,
\begin{equation}\label{fold}
\bbA^N_N[x]= \bbA^M_M[x*T^M_N].
\end{equation}
Consequently, we have the following.

\begin{lemma}\label{folding-lemma}
With $M$ and $N$ as in Definition \ref{def-folding}, 
let $\calt=(M,\vec{\delta},T^M_N)$, where 
\begin{equation}\label{deltas}
\delta_i=
\begin{cases}
	\alpha_i+1 & \hbox{ if }\alpha_i<n_i \\
	0 &\hbox{ if } \alpha_i=n_i
\end{cases}	
\qquad i\in\{1,\ldots,K\}
\end{equation}
We will sometimes write $\vec{\delta}=\vec{\delta}^M_N$ to indicate the dependence on $M$ and $N$.
Let $A\in\calm(\ZZ_M)$ be a multiset. Then the following are equivalent:
\begin{itemize} 
\item $\Phi_N|A$,
\item $A$ is $\calt$-null in $\ZZ_M$,
\item the multiset induced by $A$ in $\ZZ_N$ is $N$-null (see Definition \ref{def-N-cuboids}).
\end{itemize}
\end{lemma}

Let 
\begin{equation}\label{cuboid-lifted}
\Delta=X^c\prod_{i:p_i|N} (1-X^{d_i}), \ \ c\in\ZZ_M,\ (M,d_i)=M/p_i^{\alpha_i+1},
\end{equation}
be a cuboid of type $\calt$ as in Lemma \ref{folding-lemma}.
Then the cuboid $\Delta$ mod $N$, induced by $\Delta$ in $\ZZ_N$, is an $N$-cuboid. Conversely, any $N$-cuboid $\Delta'$ in $\ZZ_N$ can be written (not necessarily uniquely) as $\Delta$ (mod $N$), where $\Delta$ is a cuboid of the form (\ref{cuboid-lifted}) in $\ZZ_M$. Therefore, whenever working on scales $N$ and $M$ simultaneously, we will represent $N$-cuboids as cuboids of the form (\ref{cuboid-lifted}) in $\ZZ_M$. In this notation, a multiset $A\in\calm(\ZZ_M)$
satisfies any one (therefore all) of the conditions in Lemma \ref{folding-lemma}  if and only if 
$$\bbA^N_N[\Delta]= \bbA^\calt[\Delta]= \bbA^M_M[\Delta*T^M_N]=0$$
for all $\Delta$ as in (\ref{cuboid-lifted}).
Transitions between several intermediate scales $N_1,N_2,\dots |M$ will be handled similarly, with $\ZZ_M$ as the default ambient space.

Cuboids with more general templates can be used to indicate divisibility by combinations of several cyclotomic polynomials. 
We will be particularly interested in implications of the form ``if $\Phi_{s_1},\dots,\Phi_{s_l}$ divide $A(X)$, then $A$ is $\calt$-null for a given cuboid type $\calt$." It will not be necessary to aim for ``if and only if" conditions such as those in Lemma \ref{folding-lemma}.

\medskip\noindent
{\bf Examples.}
Let $M=\prod_{i=1}^K p_i^{n_i}$.

\begin{itemize}

\item[(1)] Assume that $n_i\geq 2$ for some $i\in\{1,\dots,K\}$. Let $\calt=(M, \vec{\delta}, 1)$ and $\calt'=(M, \vec{\delta}, T^M_{M/p_i})$, with $\delta_i=2$ and $\delta_j=1$ for $j\neq i$. 
Then
\begin{align*}
\Phi_{M/p_i}|A & \Leftrightarrow A \hbox{ is }\calt'\hbox{-null},\\
\Phi_M \Phi_{M/p_i}|A & \Leftrightarrow A \hbox{ is }\calt\hbox{-null}.
\end{align*}
The first equivalence follows from Lemma \ref{folding-lemma}. The second one is easy to check directly. Specifically, if $\Delta$ is a cuboid of type $\calt$, then $\Phi_m|\Delta$ for all $m|M$ except for $m\in\{M/p_i,M\}$; conversely, both $M$-cuboids and $M/p_i$-cuboids can be expressed as linear combinations of cuboids of type $\calt$. (A similar result appears in \cite[Lemma 2.13]{KMSV2}, where it is stated in terms of ``n-dimensional cube rules" and applied to Fuglede's conjecture on cyclic groups.)

\begin{figure}[h]
\captionsetup{justification=centering}
\includegraphics[scale=1.5]{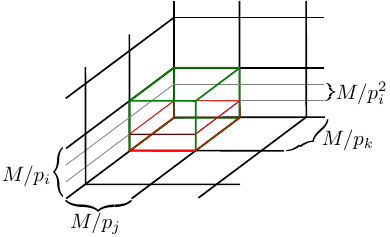}
\caption{A classic $M$-cuboid (green) vs. a multiscale cuboid (red) corresponding to the product $\Phi_M\Phi_{M/p_i}$}
\end{figure}

\item[(2)]
Let $2\leq\alpha\leq n_i$. Then 
$$\Phi_M \Phi_{M/p_i}\dots \Phi_{M/p_i^{\alpha}}|A$$
if and only if $A$ is $\calt_\alpha$-null, where $\calt_\alpha =(M, \vec{\delta}, 1)$, $\delta_i=\alpha+1$ if $\alpha_i<n_i$, $\delta_i=0$ if $\alpha_i=n_i$, and $\delta_j=1$ for $j\neq i$. This can be proved in the same way as in Example (1) above.

\item[(3)]
Assume that $n_i\geq 2$ for some $i\in\{1,\dots,K\}$. Let $\calt=(M, \vec{\delta}, T)$,
where $\delta_i=3$ if $n_i\geq 3$, $\delta_i=0$ if $n_i=2$, $\delta_j=1$ for $j\neq i$, and
$$
T(X)=\frac{X^{M/p_i}-1}{X^{M/p_i^2}-1}
= 1+X^{M/p_i^2}+\dots+X^{(p_i-1)M/p_i^2}.
$$
We claim that if $\Phi_M\Phi_{M/p_i^2}|A$, then $A$ is $\calt$-null.
Indeed, if $n_i\geq 3$, cuboids of type $\calt$ have the form 
$$
\Delta(X)=X^c(1-X^{\mu_iM/p_i^3})\prod_{j\neq i}(1-X^{\mu_jM/p_j}),
$$
 where $(\mu_i,M)=(\mu_j,M)=1$.
It follows that $\Delta(X)T(X)$ is divisible by all cyclotomic polynomials $\Phi_m(X)$, $m|M$, except for $\Phi_{M/p_i^2}$ and $\Phi_M$. If $n_i=2$, the same argument applies, except that there is no factor $1-X^{\mu_iM/p_i^3}$ in $\Delta(X)$.
\end{itemize}


\section{Tiling reductions}\label{sec-reductions}


\subsection{Subgroup reduction} 

In this section, we discuss two ways in which the question of proving (T2) for a tiling $A\oplus B=\ZZ_M$ (and, more generally, investigating the structure of such tilings) may be reduced to the analogous question for tilings $A'\oplus B'=\ZZ_{N}$, where $N|M$ and $N\neq M$. 
We start with a recap, in a slightly more general setting, of the reduction that Coven and Meyerowitz used in \cite{CM} to prove Theorem \ref{CM-thm}. 

\begin{theorem}\label{subgroup-reduction} {\bf (Subgroup reduction)} (\cite[Lemma 2.5]{CM}; see also \cite[Theorem 4.4]{dutkay-kraus})
Assume that $ A\oplus B=\ZZ_M $, where $M=\prod_{i=1}^K p_i^{n_i}$, and that:
\begin{itemize}
\item[(i)] there exists an $i\in\{1,\dots,K\}$ such that $A\subset p_i\ZZ_M$,

\item[(ii)] (T2) holds for both $A'$ and $B'$ in any tiling $A'\oplus B'=\ZZ_{N_i} $, where $N_i=M/p_i$, $|A'|=|A|$, and $|B'|=|B|/p_i$.

\end{itemize}
Then $ A $ and $ B $ satisfy (T2). 
\end{theorem}

\begin{proof}
We have $A(X)=A'(X^{p_i})$ for some $A'\subset\ZZ_{N_i}$.
Write also 
$$B(X)\equiv \sum_{\nu=0}^{p_i-1} X^{\nu M/p_i^{n_i-1}} B_\nu( X^{p_i})\hbox{ mod }X^M-1,
$$
where $B_\nu \subset\ZZ_{N_i}$ for $\nu=0,1,\dots,p_i-1$. If $b\in B$ and $b\equiv r$ mod $p_i$, then $a+b\equiv r$ mod $p_i$ for all $a\in A$; in other words, the tiling breaks down into separate tilings of residue classes mod $p_i$, with $A'\oplus B_\nu=\ZZ_{N_i}$ for each $\nu$.

By the assumption (ii), $A'$ and $B_\nu$ satisfy (T2) for all $\nu$. We need to check that this is still true for $A$ and $B$. 
We first claim that for any polynomial $F(X)$, and for any $s\in\NN$,
\begin{equation}\label{cyclo-scaling}
\Phi_{\tau(s)}(X)|F(X^{p_i}) \ \ \Leftrightarrow \ \ \Phi_s(X)|F(X),
\end{equation}
where
$$
\tau(s)=\begin{cases} s &\hbox{ if }p_i\nmid s,\\
p_is &\hbox{ if }p_i| s.
\end{cases}
$$
Indeed, we have $\Phi_{\tau(s)}(X)|F(X^{p_i})$ if and only if $F(e^{2\pi i p_i/\tau(s)})=0$. This means that 
$F(e^{2\pi i /s})=0$ if $p_i|s$, and $F(e^{2\pi i p_i/s})=0$ if $p_i\nmid s$. In both cases, this is equivalent to $\Phi_s|F$.

Observe first that, by (\ref{cyclo-uniform}), we must have
$$
\Phi_{p_i}|B.
$$
Suppose that $s_1,\dots,s_k$ are powers of distinct primes such that $\Phi_{s_1}\dots\Phi_{s_k}|A$. As noted above, we cannot have $s_j=p_i$ for any $j$. Let $s'_j=s_j/p_i$ if $s_j$ is a power of $p_i$, and $s'_j=s_j$ otherwise. Then
$s'_j$ are prime powers, and $s_j=\tau(s'_j)$. By (\ref{cyclo-scaling}), $\Phi_{s'_1}\dots\Phi_{s'_k}|A'$, and
since (T2) holds for $A'$, we have $\Phi_{s'_1\dots s'_k}|A'$. Since $\tau(s'_1\dots s'_k)=s_1\dots s_k$, we get that
$\Phi_{s_1\dots s_k}|A$.

Suppose now that $s_1,\dots,s_k$ are powers of distinct primes such that $\Phi_{s_1}\dots\Phi_{s_k}|B$ and $s_1,\dots,s_k\neq p_i$, and define $s'_1,\dots,s'_k$ as above. Then for $j=1,\dots,k$ we have $\Phi_{s_j}\nmid A$, therefore $\Phi_{s'_j}\nmid A'$ and, since $A'\oplus B_\nu=\ZZ_{N_i}$ is a tiling, $\Phi_{s'_j}|B_\nu$ for each $\nu$. Since $B_\nu$ satisfies (T2), we have 
$\Phi_{s'_1\dots s'_k}|B_\nu$. It follows that $\Phi_{s_1\dots s_k}|B_\nu(X^{p_i})$ for each $\nu$, and therefore
$\Phi_{s_1\dots s_k}|B$.

Finally, suppose that $s_1,\dots,s_k$ are powers of distinct primes such that $\Phi_{s_1}\dots\Phi_{s_k}|B$ and $s_1,\dots,s_k$ are not powers of $p_i$, and consider $\Phi_{p_is}$ with $s=s_1\dots s_k$. We have
$$
B (e^{2\pi i/p_is})= \sum_{\nu=0}^{p_i-1} e^{2\pi i \nu M/sp_i^{n_i}} B_\nu( e^{2\pi is})=0,
$$
by (T2) for each $B_\nu$.
\end{proof}

\begin{corollary}\label{almostsquarefree} 
Let $ A\oplus B=\ZZ_M $, where $M=p_1^{n_1}p_2^{n_2}p_3^{n_3} \dots p_K^{n_K}$.
Assume that for each $i\geq 3$, the prime factor $p_i$ divides at most one of $|A|$ and $|B|$.  
(This happens e.g., if $n_i=1$ for $i\not\in\{1,2\}$).
Then both $A$ and $B$ satisfy (T2). 
\end{corollary}

\begin{proof}
This is not stated explicitly in \cite{CM}, but it follows by a very similar argument. 
(See also \cite{Tao-blog}, \cite{dutkay-kraus}, \cite{shi}.)
We proceed by induction in the number of prime factors. If $K=2$ and $M=p_1^{n_1}p_2^{n_2}$, this is Theorem \ref{CM-thm}. Suppose that $K\geq 3$ and that 
(T2) holds for both $A'$ and $B'$ in any tiling $ A'\oplus B'=\ZZ_{M/p_K}$. By the assumption of the lemma, at least one of $|A|$ and $|B|$ is not divisible by $p_K$. Assume without loss of generality that $p_K\nmid |A|$. By Tijdeman's theorem (\cite[Theorem 1]{Tij}; see also \cite[Lemma 2.2]{CM}), $\tilde A\oplus B=\ZZ_M$ is again a tiling, where $\tilde A(X)=A(X^{p_K})$. We have $\tilde A\subset p_K\ZZ_M$, so that we may apply Theorem \ref{subgroup-reduction} to conclude that $\tilde A$ and $B$ satisfy (T2). By (\ref{cyclo-scaling}), this also means that $A$ satisfies (T2), since the (T2) condition for $A$ involves only cyclotomic polynomials $\Phi_s$ with $(s,p_K)=1$.
\end{proof}


\subsection{Slab reduction}

Our second tiling reduction also involves passing from a tiling $A\oplus B=\ZZ_M$ to a tiling of a smaller cyclic group. However, instead of restricting to residue classes and thus constructing a family of tilings of a subgroup $p_i\ZZ_M$, we will use periodicity. 
Recall that $M$-fibers $F_i$ and $M$-fibered sets were defined in Section \ref{gridsetal} (see (\ref{def-Fi})).
Let $M=\prod_{i=1}^K p_i^{n_i}$, and define 
\begin{equation}\label{Asubtile}
A_{p_i}=\{a\in A:\ 0\leq\pi_i(a)\leq p_i^{n_i-1}-1\},
\end{equation}
where $\pi_i$ is the array coordinate defined in Section \ref{sec-array}. Suppose that we have $S\oplus B=\ZZ_M$, where $S$ is the $M/p_i$-periodic extension of $A_{p_i}$ to $\ZZ_M$:
\begin{equation}\label{periodic-extension}
S(X)=A_{p_i}(X)F_i(X) .
\end{equation}
Then we may reduce the period of the tiling and write $A_{p_i}\oplus B=\ZZ_{M/p_i}$, where $A_{p_i}$ and $B$ are now considered mod $M/p_i$.

As a motivating example, suppose that $A\oplus B=\ZZ_M$, with $M$ as above, and that $A$ is $M$-fibered in the $p_i$ direction.
Let $A'$ be a set obtained from $A$ by choosing one point from each fiber, so that $|A'|=|A|/p_i$ and $A=A'*F_i$. Then $A$ is the periodic extension of $A'$, and we have $A'\oplus B=\ZZ_{M/p_i}$.

Our main results in this section are Theorem \ref{subtile} and Corollary \ref{slab-reduction}, where we develop a criterion for $A_{p_i}$ to admit periodic tilings as described above, and prove that passing to such tilings preserves the (T2) property.

\begin{lemma}\label{precycreduc}
Let $ A\in\mathcal{M}(\ZZ_{M}) $, with $M$ as above. Assume that $\Phi_d|A$ for some $d$ such that 
$p_i^{n_i}|d|M$. Then for every $1\leq \alpha_i\leq n_i $, 
$$ \Phi_{d/p_i^{\alpha_i}}|A \Rightarrow \Phi_{d/p_i^{\alpha_i}}|A_{p_i}.$$
\end{lemma}

\begin{proof}
Let $d = M/\prod_{j\neq i} p_j^{\alpha_j}$ and $d'=d/p_i^{\alpha_i}$. Assume that $\Phi_d\Phi_{d'}|A$. 
We would like to show that $\Phi_{d'}|A_{p_i}$. To this end, we define cuboid types $\calt=(M,\vec\delta,T_d)$ and 
$\calt'=(M,\vec\delta',T_{d'})$, where $T_d=T^M_d$, $T_{d'}=T^M_{d'}$ are the folding templates from Definition \ref{def-folding}, and $\vec{\delta}=\vec{\delta}^M_d$, $\vec{\delta}'=\vec{\delta}^M_{d'}$
are defined as in (\ref{deltas}) with $N=d$ and $N=d'$, respectively.

Let $S$ be the periodic extension of $A_{p_i}$ to $\ZZ_M$ defined in (\ref{periodic-extension}).
We have $\Phi_d| F_i$ but $\Phi_{d'}\nmid F_i$, so that $\Phi_{d'}|S$ if and only if $\Phi_{d'}|A_{p_i}$. We need to prove that
$$
\mathbb{S}^{d'}_{d'}[\Delta]=0
$$
for all cuboids $\Delta$ of type $\calt'$. Fix such a cuboid
$$
\Delta(X)=X^y\cdot\prod_j (1-X^{d_j}), \text{ where } y\in\ZZ_M, (d_j,M)=M/p_{j}^{\delta'_j}.
$$
Let
$$
\Delta_i(X)=X^y\cdot\prod_{j\neq i} (1-X^{d_j})
$$
so that 
$$\Delta(X)=\begin{cases}
(1-X^{d_i})\Delta_i(X)&\hbox{ if }\alpha_i<n_i,\\
\Delta_i(X)&\hbox{ if }\alpha_i=n_i.
\end{cases}
$$

Observe that if  $\rho_i\in\ZZ_M$ satisfies $(\rho_i,M)=M/p_i$, then $(1-X^{\rho_i})\Delta_i(X)$  is a cuboid of type $\calt$. Since $A$ is $\calt$-null, we have $\bbA^d_d[(1-X^{\rho_i})\Delta_i(X)]=0$, so that 
$$
\bbA^d_d[\Delta_i]=\bbA^d_d[\rho_i*\Delta_i].
$$
Averaging the last equality over all $\rho_i\in \{M/p_i,2M/p_i,\ldots,(p_i-1)M/p_i\}$, we get 
$$
\bbA^d_d[\Delta_i]=\frac{1}{\phi(p_i)}\bbA^d_d[\Delta_i*(F_i-1)].
$$
Clearly we may take linear combinations of the latter, i.e. for any set $V\subset\ZZ_M$,
\begin{equation}\label{cuboid_av}
\bbA^d_d[\Delta_i*V]=\frac{1}{\phi(p_i)}\bbA^d_d[\Delta_i*(F_i-1)*V].
\end{equation}

Let $\Psi\subset\ZZ_M$ be a set such that $\Psi(X)\equiv \prod_{\nu=2}^{\alpha_i}\Psi_{M/p_i^\nu}(X)$ mod $X^{M/p_i}-1$, and 
\begin{equation}\label{target-slab}
0\leq\pi_i(y+z) \leq p_i^{n_i-1}-1\ \ \forall z\in \Psi.
\end{equation}
Then
\begin{align*}
T_{d'}(X)&= T_d(X)F_i(X)\Psi(X)\\
&= T_d(X)\Psi(X) +(F_i-1)T_d(X)\Psi(X) ,
\end{align*}	
so that
\begin{align*}
\bbA^{d'}_{d'}[\Delta_i]&= \bbA^M_M[\Delta_i*T_{d'}]\\
&=\bbA^M_M[\Delta_i*T_d*\Psi ]
+\bbA^M_M[\Delta_i*(F_i-1)*T_d*\Psi]\\
&=\bbA^d_d[\Delta_i*\Psi]+\bbA^d_d[\Delta_i*(F_i-1)*\Psi]
\end{align*}
By (\ref{cuboid_av}) with $V=\Psi$, we get 
\begin{equation}\label{p-distribute}
\bbA^{d'}_{d'}[\Delta_i]=\frac{p_i}{\phi(p_i)}\bbA^d_d[\Delta_i*\Psi].
\end{equation}

The proof of (\ref{p-distribute}) used only that $\Phi_d|A$. Since $S$ has the same property, it follows that
$$
\mathbb{S}^{d'}_{d'}[\Delta_i]=\frac{p_i}{\phi(p_i)}\mathbb{S}^d_d[\Delta_i*\Psi],
$$
By (\ref{target-slab}), we also have
$$
\bbA^d_d[\Delta_i*\Psi]=\mathbb{S}^d_d[\Delta_i*\Psi].
$$

Assume first that $\alpha_i<n_i$, and note that all of the above arguments apply with $\Delta_i$ replaced by $d_i*\Delta_i$, yielding the same conclusions with $\Psi$ replaced by $\Psi'$ such that (\ref{target-slab}) holds with $y$ replaced by $y+d_i$.
Recall that $A$ is $\calt'$-null, so that
\begin{align*}
0&=\bbA^{d'}_{d'}[\Delta]\\
&= \bbA^{d'}_{d'}[\Delta_i]-\bbA^{d'}_{d'}[{d_i}*\Delta_i]\\
&= \frac{p_i}{\phi(p_i)}(\bbA^d_d[\Delta_i*\Psi]-\bbA^d_d[{d_i}*\Delta_i*\Psi'])\\
&=\frac{p_i}{\phi(p_i)}(\mathbb{S}^d_d[\Delta_i*\Psi]-\mathbb{S}^d_d[{d_i}*\Delta_i*\Psi']).
\end{align*}
Taking the convolution with $\Psi_{M/p_i}$, and using that $\Psi(X)\Psi_{M/p_i}(X)\equiv \Psi'(X)\Psi_{M/p_i}(X)$ mod $(X^M-1)$, we conclude that
\begin{align*}
0 &= \mathbb{S}^d_d[\Delta*\Psi* \Psi_{M/p_i}]\\
&= \mathbb{S}^{d'}_{d'}[\Delta]
\end{align*}
It follows that $S$ is $\calt'$-null, as required.

If $\alpha_i=n_i$, the proof is the same as above, except that then $\Delta=\Delta_i$ and so the terms with $d_i*\Delta_i$ do not appear in the above calculation.
\end{proof}

\begin{lemma}\label{cycreduc}
Let $ A\in\mathcal{M}(\ZZ_{M}) $, with $M$ as above, and let $p_i^{n_i}|d|M$ and $ 1\leq\alpha_i\leq n_i $.
Assume that $\Phi_{d/p_i^{\alpha_i}}|A'_{p_i}$ for all translates $A'$ of $A$. Then $\Phi_d\Phi_{d/p_i^{\alpha_i}}|A$.
\end{lemma}

\begin{proof}
Let $d = M/\prod_{j\neq i} p_j^{\alpha_j}$ and $d'=d/p_i^{\alpha_i}$. Define the cuboid types $\calt=(M,\vec\delta,T_d)$ and $\calt'=(M,\vec\delta',T_{d'})$ as in the proof of Lemma \ref{precycreduc}.

In order to prove that $\Phi_{d'}|A$, it suffices to show $ \bbA^{d'}_{d'}[\Delta]=0$ for all cuboids of the form
\begin{equation}\label{dcuboid}
\Delta'(X)=X^y\cdot\prod_j (1-X^{d_j}), \ \  y\in\ZZ_M, (d_j,M)=M/p_{j}^{\delta'_j}.
\end{equation} 
We write $ A(X)=\sum_{\nu=0}^{p_i-1} A_{\nu}(X) $, where 
\begin{equation}\label{subsets}
A_{\nu}=\{a\in A\,|\,0\leq\pi_i(a)-\nu p_i^{n_i-1}\leq p_i^{n_i-1}-1\},\quad\nu=0,1,\ldots, p_i-1.
\end{equation}
Then
$$
\bbA^{d'}_{d'}[\Delta']=\sum_{\nu=0}^{p_i-1}(\bbA_\nu)^{d'}_{d'}[\Delta'] =0
$$
using the assumption that $A_\nu$ are $\calt'$-null for all $\nu$.

We now prove that $ \Phi_d|A $. It suffices to show $\bbA^d_d[\Delta]=0$ for any cuboid of the form
\begin{equation}\label{longcuboid}
\Delta(X)=(1-X^{M/p_i})\Delta_i(X),
\end{equation}
where $y\in\ZZ_N$, and 
$$\Delta_i(X)=X^y\cdot\prod_{j\neq i} (1-X^{d_j}), \ \  y\in\ZZ_M, (d_j,M)=M/p_{j}^{\delta_j}.
$$
Indeed, any cuboid of type $\calt$ can be written as a linear combination of cuboids as in (\ref{longcuboid}).

Let $\Psi(X)= \prod_{\nu=2}^{\alpha_i}\Psi_{M/p_i^\nu}(X)$, and
$$\Delta''(X)=\begin{cases}
\big(1-X^{M/p_i^{\alpha_i+1}}\big)\Delta_i(X)&\hbox{ if }\alpha_i<n_i,\\
\Delta_i(X)&\hbox{ if }\alpha_i=n_i.
\end{cases}
$$
Define also
\begin{align*}
A'_i&=\{a\in A:\ 0\leq\pi_i(a-y)\leq p_i^{n_i-1}-1\},\\
A''_i&=\{a\in A:\ 1\leq\pi_i(a-y)\leq p_i^{n_i-1}\}.
\end{align*}
By our assumption on $A$, $ \Phi_{d'} $ divides both $A'_i$ and $A''_i$. Suppose first that $\alpha_i<n_i$. Then
\begin{align*}
0&= (\bbA'_i)^{d'}_{d'}[\Delta''] = \bbA^{d}_{d}[\Delta_i*\Psi] - \bbA^{d}_{d}[(M/p_i^{\alpha_i+1})*\Delta_i*\Psi],
\\
0&= (\bbA''_i)^{d'}_{d'}[\Delta''] = \bbA^{d}_{d}[\Delta_i*\Psi'] - \bbA^{d}_{d}[(M/p_i^{\alpha_i+1})*\Delta_i*\Psi],
\end{align*}
where $\Psi'(X)=\Psi-1+X^{M/p_i}$.
Taking the difference, we get
\begin{align*}
0&=(\bbA'_i)^{d'}_{d'}[\Delta''] - (\bbA''_i)^{d'}_{d'}[\Delta''] 
\\
&= \bbA^{d}_{d}[\Delta_i*\Psi] -  \bbA^{d}_{d}[\Delta_i*\Psi']
\\
& = \bbA^{d}_{d}[\Delta],
\end{align*}
which proves the claim.

If $\alpha_i=n_i$, the proof is the same, except that the terms with $(M/p_i^{\alpha_i+1})*\Delta_i$ are replaced by $0$ in the above calculation.
\end{proof}

\begin{theorem}\label{subtile}
Assume that $A\oplus B=\ZZ_M$ and $\Phi_{p_i^{n_i}}|A$. Then the following are equivalent:

\smallskip

(i) For any translate $ A' $ of $ A $ we have $A'_{p_i}\oplus B=\ZZ_{M/p_i}$.
 
\smallskip

(ii) 
For every $d$ such that $p_i^{n_i}|d|M$, at least one of the following holds:
\begin{equation}\label{ts-e2}
\Phi_d|A,
\end{equation}
\begin{equation}\label{ts-e1}
\Phi_{d/p_i}\Phi_{d/p_i^2}\dots \Phi_{d/p_i^{n_i}}\mid B.
\end{equation}

\smallskip

(iii) 
For every $p_i^{n_i}|m|M$,
\begin{equation}\label{ts-e3}
 m\in \Div(A) \Rightarrow m/p_i\notin \Div(B).
\end{equation}
\end{theorem}

\begin{proof} Let $N_i=M/p_i$. The assumption that $\Phi_{p_i^{n_i}}|A$ implies that for any translate $A'$ of $A$,
we have $|A'_{p_i}|=|A|/p_i$,  so that 
\begin{equation}\label{ts-e4}
|A'_{p_i}|\,|B|=N_i.
\end{equation}
 
(i) $\Rightarrow$ (ii): Assume that (i) holds, and suppose that (\ref{ts-e1}) fails for some $p_i^{n_i}|d|M$. Then there is an $\alpha_i$ such that $ 1\leq\alpha_i\leq n_i$ and $\Phi_{d/p_i^{\alpha_i}}\nmid B$. Then $\Phi_{d/p_i^{\alpha_i}}|A'_{p_i}$ for any translate $ A' $ of $ A $. By Lemma \ref{cycreduc}, (\ref{ts-e2}) must hold.

\smallskip

(ii) $\Rightarrow$ (i): Assume that (ii) holds. With (\ref{ts-e4}) in place, it suffices to prove that 
for every $ d'|N_i,\,d'>1 $,  $ \Phi_{d'}$ divides at least one of $A'_{p_i}(X)$ and $B(X)$. Let $d'|N_i$, $d'>1$, and suppose that $\Phi_{d'}\nmid B$.  Then $d'=d/p_i^{\alpha_i}$ for some $ p_i^{n_i}|d|M$ and $ 1\leq\alpha_i\leq n_i$. By (ii), we must have $\Phi_d|A$. Since $A\oplus B=\ZZ_M$ is a tiling, we must also have $\Phi_{d'}|A$. By Lemma \ref{precycreduc}, we must have $\Phi_{d'}|A'_{p_i}$ as claimed.

\smallskip

(i) $\Rightarrow$ (iii): Assume that (i) holds. This implies in particular that $A'_{p_i}$ and $B$ mod $\ZZ_{N_i}$ are sets, so that $N_i\not\in\Div(B)$ and (\ref{ts-e3}) holds for $m=M$. 

Next, Theorem \ref{thm-sands} applied to the tiling $A'_{p_i}\oplus B=\ZZ_{N_i}$ implies that
\begin{equation}\label{ts-e5}
 \Div_{N_i}(A) \cap \Div_{N_i}(B)= \{N_i\},
\end{equation}
Suppose that (\ref{ts-e3}) fails for some $m\neq M$ such that $p_i^{n_i}|m|M$, so that 
$m\in \Div(A)$ and $m/p_i\in \Div(B)$. Then $m/p_i \neq N_i$ and $m/p_i\in\Div_{N_i}(A) \cap \Div_{N_i}(B)$.
But this contradicts (\ref{ts-e5}).

\smallskip

(iii) $\Rightarrow$ (i): Assume that (iii) holds. By Theorem \ref{thm-sands}, it suffices to prove that $A'_{p_i}, B$ mod $\ZZ_{N_i}$ are sets such that (\ref{ts-e5}) holds.

We first verify that $A'_{p_i}, B$ mod $\ZZ_{N_i}$ are sets. Indeed, if $a,a'\in A'_{p_i}$ and $a\equiv a'$ mod $N_i$, then $a=a'$ by the definition of $A'_{p_i}$. On the other hand, $M\in\Div(A)$ trivially, and by (\ref{ts-e3}) it follows that $N_i\not\in\Div(B)$, so that $B$ mod $N_i$ is also a set.

Suppose now that (\ref{ts-e5}) fails, with $m_1\in (\Div_{N_i}(A) \cap \Div_{N_i}(B))\setminus \{N_i\}$. Since $\Div(A)\cap\Div(B)=\{M\}$, we must have $m_1=m_2/p_i$ for some $m_2$ with $p_i^{n_i}|m_2|M$, so that
$$
m_{j}\in\Div(A'_{p_i}),\ m_{k}\in\Div(B)
$$
for some permutation $(j,k)$ of $(1,2)$. By the definition of $A'_{p_i}$, we cannot have $p_i^{n_i-1}\parallel s$ for
$s\in \Div(A'_{p_i})$, so that $j=2$, $k=1$. But this contradicts (\ref{ts-e3}).
\end{proof}

\begin{remark}\label{fiberingimpliesslab}
In the special case when $A$ is $M$-fibered in the $ p_i$ direction, the condition (\ref{ts-e1})
 of Theorem \ref{subtile} is satisfied since then $ \Phi_d|A$ for all $p_i^{n_i}|d|M$.
It is also easy to verify directly that (\ref{ts-e3}) holds in this case.

\end{remark}

\begin{corollary}\label{slab-reduction} {\bf (Slab reduction)}
Assume that $ A\oplus B=\ZZ_M $, where $M=\prod_{i=1}^K p_i^{n_i}$, and that:
\begin{itemize}
\item (T2) holds for both $A'$ and $B'$ in any tiling $ A'\oplus B'=\ZZ_{N_i} $, where $N_i=M/p_i$, $|A'|=|A|/p_i$, and $|B'|=|B|$,
\item there exists an $i\in\{1,\dots,K\}$ such that $\Phi_{p_i^{n_i}}|A$ and $A,B$ obey one (therefore all) of the conditions (i)-(iii) of Theorem \ref{subtile}.
\end{itemize}
Then $ A $ and $ B $ satisfy (T2). 
\end{corollary}

\begin{proof}
We are assuming that $A'_{p_i}\oplus B=\ZZ_{M/p_i}$ for any translate $ A'$ of $ A $. By the inductive part of the assumption, $A'_{p_i}$ and $B$ satisfy (T2). It remains to prove (T2) for $ A $. Suppose that 
$d=\prod_{j\in J} p_j^{\alpha_j}$, where $J\subset \{1,\dots,K\}$, $1\leq \alpha_j\leq n_j$ for all $j\in J$, and
$\Phi_{p_j^{\alpha_j}}(X)|A(X)$ for all $j\in J$. 
For each prime power $p_j^{\alpha_j}$ with $\alpha_j\neq 0$, the polynomial $\Phi_{p_j^{\alpha_j}}(X)$ can divide only one of $A$ and $B$ in the tiling $A\oplus B=\ZZ_M$, hence 
\begin{equation}\label{ts-e6}
\Phi_{p_j^{\alpha_j}}\nmid B \ \ \forall j\in J.
\end{equation}

Write $ A(X)=\sum_{\nu=0}^{p_i-1} A_{\nu}(X) $, where $A_\nu$ are as in (\ref{subsets}), so that $A_\nu\oplus B=\ZZ_{N_i}$ for each $\nu$. 
We consider two cases.

\begin{itemize}
\item Assume that either $i\not\in J$, or $i\in J$ but $\alpha_i\neq n_i$. By (\ref{ts-e6}), we have $\Phi_{p_j^{\alpha_j}}|A_\nu$ for all $j\in J$ and $\nu=0,1,\dots,p_i-1$. We are assuming that (T2) holds for $A_\nu$, so that $\Phi_d|A_\nu$. Summing over $\nu$, we get that $\Phi_d|A$.

\item Assume now that $i\in J$ and $\alpha_i=n_i$, and let $d'=d/p_i^{n_i}$. Then $\Phi_{d'}|A'_{p_i}$ for any translate $A'$ of $A$, by the argument in the first case applied to $A'$ instead of $A$. By Lemma \ref{cycreduc}, it follows that $\Phi_d|A$.
\end{itemize}
\end{proof}

We note the following special case.

\begin{corollary}\label{Afibered} 
Assume that $ A\oplus B=\ZZ_M $, where $M=p_1^{n_1}p_2^{n_2}p_3^{n_3}$ and $p_1,p_2,p_3$ are distinct primes. 
Moreover, assume that there is a permutation $(i,j,k)$ of $(1,2,3)$ such that $|A|=p_i p_j^{\alpha_j} p_k^{\alpha_k}$ for some $0\leq\alpha_j\leq n_j$, $0\leq\alpha_k\leq n_k$, and that $A$ is $M$-fibered in the $p_i$ direction.  
Then $ A $ and $ B $ satisfy (T2). 
\end{corollary}

\begin{proof}
This follows from Corollary \ref{slab-reduction} and Corollary \ref{almostsquarefree}.
\end{proof}


\section{Saturating sets}\label{sec-satsets}

\subsection{Preliminaries}
\begin{definition}\label{restricted-boxes} {\bf (Restricted $N$-boxes)}
Let $A,X\subseteq\ZZ_M$, and $x\in\ZZ_M$.
The {\em restriction of $\bbA^N[x]$ to $X$} is the $N$-box $\bbA^N[x|X]$  
with entries
$$
\bbA^N_m[x|X] = \sum_{a\in  X: \ (x-a,N)=m} w^N_A(a), \ \ m|N.
$$
In particular,
$$
\bbA^M_m[x|X] = \# \{a\in A\cap X: \ (x-a,M)=m\}.
$$
\end{definition}

The next definition is the key to our analysis of unfibered tilings in \cite{LaLo2}. While 
it could be extended in an obvious way to $N$-boxes with $N|M$, our current arguments only use the $M$-box version below.

\begin{definition}\label{saturating2} {\bf (Saturating sets)}
Let $A,B\subseteq \ZZ_M$, and $x,y\in\ZZ_M$.
Define
$$
A_{x,y}:=\{a\in A:\ (x-a,M)=(y-b,M) \hbox{ for some }b\in B\},
$$
$$
A_{x}:=\bigcup_{b\in B} A_{x,b}.
$$
Equivalently,
\begin{equation}\label{saturate-e6}
A_x=\{a\in A: (x-a,M)\in\Div(B)\}.
\end{equation}
We will refer to $A_x$ as the {\em saturating set} for $x$. The sets $B_{y,x}$ and $B_y$ are defined similarly, with $A$ and $B$ interchanged.

\end{definition}

With the above notation, $A_{x,y}$ is the minimal set that saturates 
(the $A$-side of) the product $\langle \bbA^M[x], \bbB^M[y]\rangle$, in the sense that
\begin{equation}\label{saturate-e1}
\langle \bbA^M[x|X], \bbB^M[y]\rangle = \langle \bbA^M[x], \bbB^M[y]\rangle
\end{equation}
holds for $X=A_{x,y}$, and if $X\subset\ZZ_M$ is any other set for which (\ref{saturate-e1}) holds, then $A_{x,y}\subset X$. 
The set $A_x$ is the minimal set such that
\begin{equation}\label{saturate-e2}
\langle \bbA^M[x|A_x], \bbB^M[b]\rangle =\langle \bbA^M[x], \bbB^M[b]\rangle
\ \ \forall b\in B.
\end{equation}

While Definition \ref{saturating2} makes sense for general sets $A,B\subseteq \ZZ_M$, our intended application is
to the tiling situation $A\oplus B= \ZZ_M$. In that case, 
by Theorem \ref{ortho-lemma}, the box products on the right side of (\ref{saturate-e1}) and (\ref{saturate-e2}) evaluate to 1. Hence $A_{x,y}$ is the smallest set such that
\begin{equation}\label{saturate-e3}
\langle \bbA^M[x|A_{x,y}], \bbB^M[y]\rangle = 1,
\end{equation}
and $A_x$ is the smallest set such that
\begin{equation}\label{saturate-e4}
\langle \bbA^M[x|A_x], \bbB^M[b]\rangle =1
\ \ \forall b\in B.
\end{equation} 
Observe in particular that a saturating set for any $x\in\ZZ_M$ must be nonempty, and that by Theorem \ref{thm-sands},
\begin{equation}\label{saturate-e5}
A_a=\{a\} \ \ \ \forall a\in A.
\end{equation}

In the next few definitions and lemmas, we will work towards geometric descriptions of saturating sets. Assume that $A\oplus B= \ZZ_M$. Let $x\in\ZZ_M\setminus A$, and suppose that $a\in A_x$. By (\ref{saturate-e6}) and divisor exclusion, we must have $(x-a,M)\not\in\Div(A)$, and in particular $(x-a,M)\neq (a'-a,M)$ for all $a'\in A$. This motivates the definition below.


\begin{definition}\label{def-span}
Let $M=p_1^{n_1}\dots p_K^{n_K}$, where $p_1,\dots,p_K$ are distinct primes,
and let
$x,x'\in\ZZ_M$, $x\neq x'$. Suppose that $(x-x',M)=p_1^{\alpha_1}\dots p_K^{\alpha_K}$, with $0\leq \alpha_j\leq n_j$ for $j=1,\dots,K$.

Define
\begin{equation}\label{e-span}
\begin{split}
\Span(x,x')&=\bigcup_{i: \alpha_i<n_i} \Pi(x,p_i^{\alpha_i+1}),
\\
\Bispan(x,x')&= \Span(x,x')\cup \Span(x',x).
\end{split}
\end{equation}

\end{definition}

\medskip
\noindent
{\bf Examples:} Let $M=p_1^{n_1}p_2^{n_2}p_3^{n_3}$, where $p_1,p_2,p_3$ are distinct primes and $n_1,n_2,n_3\geq 2$. Let $x,x'\in \ZZ_M$, and let $m=(x-x',M)$.

\begin{itemize}
\item Suppose that $m=M/p_1p_2p_3$, so that $\alpha_i=n_i-1$ for $i=1,2,3$, and represent $\ZZ_M$ as a 3-dimensional $M$-array. Then $\Span(x,x')$ is the union of the 2-dimensional planes 
$\Pi(x,p_i^{n_i})$ with $i=1,2,3$, all passing through $x$, and similarly 
for $\Span(x',x)$, with $x$ and $x'$ interchanged.
Geometrically, $\Bispan(x,x')$ is the union of those 2-dimensional planes at the top scale that contain at least one 
2-dimensional face of the 3-dimensional rectangular box ``spanned" by $x$ and $x'$.

\item Suppose now that $m=M/p_i$ for some $i\in\{1,2,3\}$. Then 
$\Span(x,x')=\Pi(x,p_i^{n_i})$ is a single plane passing through $x$ and perpendicular to the $p_i$ direction, and, similarly, $\Span(x',x)=\Pi(x',p_i^{n_i})$. 

\item If $m=M/p_ip_j$ for some $i\neq j$, then $\Span(x,x')=\Pi(x,p_i^{n_i})\cup \Pi(x,p_j^{n_j})$ is a union of two planes.
\end{itemize}

The higher-dimensional case has a similar interpretation.
As should be clear from the above examples, 
the definition is not symmetric with respect to $x,x'$, so that $\Span(x,x')\neq \Span(x',x)$. 
However, we have the following.

\begin{lemma}\label{span-exchange}
Let $x,x'\in \ZZ_M$. Then 
$$x'\in \Span(x,z) \Leftrightarrow x\in  \Span(x',z).$$
\end{lemma}

\begin{proof}
We have $x'\in \Span(x,z)$ if and only if there exist $i\in \{1,\dots,K\}$ 
and $0\leq \alpha_i<n_i$ such that $p_i^{\alpha_i}\parallel x-z$, $p_i^{\alpha_i} \parallel x'-z$,
and $p_i^{\alpha_i+1}\,|\, x-x'$. These conditions are clearly symmetric with respect to $x$ and $x'$.
\end{proof}

\begin{lemma}\label{span-lemma2}
If $(x-x',M)=(x-x'',M)=m$, then $\Span(x,x')=\Span(x,x'')$.
\end{lemma}

\begin{proof}
This follows directly from the definition.
\end{proof}

\begin{lemma}\label{span-lemma}
Let $A,B\subset\ZZ_M$, $x,x'\in \ZZ_M$, and $a\in A$. If $(x-a,M) \neq (x'-a,M)$,
then $a\in \Bispan(x,x')$.
\end{lemma}

\begin{proof}
Suppose that $(x-a,M)\neq (x'-a,M)$. It follows that $(x-a,p_i^{n_i})\neq (x'-a,p_i^{n_i})$ for some 
$i\in\{1,\dots,K\}$. Interchanging $x$ and $x'$ if necessary, we may assume that
$p_i^{\alpha_i}\parallel x-a$ and $p_i^{\alpha_i+1}\,|\, x'-a$ for some 
$\alpha_i\in\{0,1,\dots, n_i-1\}$. Hence $p_i^{\alpha_i}\parallel x-x'$, and 
$a\in\Span (x',x)$.
\end{proof}

\begin{lemma}\label{saturating-basics}
Let $A,B\subset\ZZ_M$ be fixed. Then: 
\smallskip

(i) For any $x,x',y\in\ZZ_M$, we have
\begin{equation}\label{setplusspan}
 A_{x',y}\subset A_{x,y}\cup\Bispan(x,x').
\end{equation}
\smallskip

(ii) Suppose that $A\oplus B=\ZZ_M$. Then for any $x\in\ZZ_M$,
\begin{equation}\label{bispan}
A_x \subset \bigcap_{a\in A} \Bispan (x,a).
\end{equation}

\end{lemma}

\begin{proof}
To prove (i), suppose that $a\in A_{x',y}$. Then $(x'-a,M)=(y-b,M)$ for some $b\in B$. If $(x-a,M)=(x'-a,M)$, it follows that $a\in A_{x,y}$. If on the other hand $(x-a,M)\neq(x'-a,M)$, then by Lemma \ref{span-lemma} we must have $a\in\Bispan(x,x')$. This proves (\ref{setplusspan}). Part (ii) follows from (i) and (\ref{saturate-e5}), since $a\in\Span(a,x)$.
\end{proof}

We note a lemma which will be useful in the evaluation of saturating sets.

\begin{lemma}\label{triangles} {\bf (Enhanced divisor exclusion)}
Let $A\oplus B=\ZZ_M$, with $M=\prod_{i=1}^K p_i^{n_i}$. Let $m=\prod_{i=1}^K p_i^{\alpha_i}$
and $m'=\prod_{i=1}^K p_i^{\alpha'_i}$, with $0\leq \alpha_i,\alpha'_i\leq n_i$. Assume that at least
one of $m,m'$ is different from $M$, and that for every $i=1,\dots,K$ we have either $\alpha_i\neq \alpha'_i$
or $\alpha_i =\alpha'_i=n_i$. Then for all $x,y\in\ZZ_M$ we have
$$
\bbA^M_m[x] \,\bbA^M_{m'}[x] \, \bbB^M_m[y] \, \bbB^M_{m'}[y] =0.
$$
In other words, there are no configurations $(a,a',b,b')\in A\times A\times B\times B$ such that
\begin{equation}\label{step}
(a-x,M)=(b-y,M)=m,\ \ (a'-x,M)=(b'-y,M)=m'.
\end{equation}
\end{lemma}

\begin{proof}
If we did have a configuration as in (\ref{step}), then, under the assumptions of the lemma,
we would have
$$
(a-a',M)=(b-b',M)=\prod_{i=1}^K p_i^{\min(\alpha_i,\alpha'_i)},
$$
with the right side different from $M$. But that is prohibited by Theorem \ref{thm-sands}.
\end{proof}

\subsection{Examples and applications}

We first provide examples of using Lemma \ref{saturating-basics} to derive
geometric constraints on saturating sets. 
For simplicity, in the examples below we return to the 3-prime case with $M=p_1^{n_1}p_2^{n_2}p_3^{n_3}$, where $p_1,p_2,p_3$ are distinct primes and $n_1,n_2,n_3\geq 2$. Assume that $A\oplus B=\ZZ_M$ is a tiling, and let $x\in \ZZ_M$. 


\begin{itemize}
	\item Suppose that $(x-a,M)=M/p_i$ for some $a\in A$ and $i\in\{1,2,3\}$. Then
	$$
	A_x\subset \Bispan(x,a)=\Pi(x,p_i^{n_i}) \cup\Pi(a,p_i^{n_i}). 
	$$
	\setlength\intextsep{1pt}
	\begin{figure}[h]
		\includegraphics[scale=0.75]{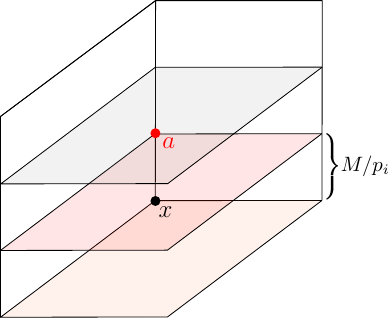}
	\end{figure}

	\item Suppose that there are two distinct elements $a,a'\in A$ such that $(x-a,M)=(x-a',M)=M/p_i$. Then
	$$
	A_x\subset \Bispan(x,a)\cap\Bispan(x,a')=\Pi(x,p_i^{n_i}) .
	$$
	
	\setlength\intextsep{1pt}
	\begin{figure}[h]
		\includegraphics[scale=0.75]{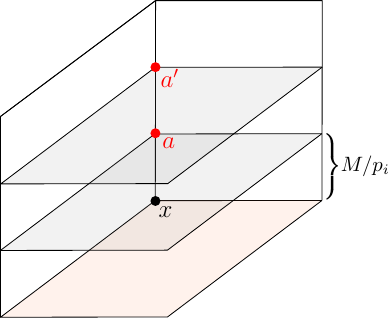}
	\end{figure}
	\item Suppose that there are two elements $a_i,a_j\in A$ such that $(x-a_i,M)=M/p_i$ and $(x-a_j,M)=M/p_j$, with $i,j\in\{1,2,3\}$ distinct.
	Then
	$$
	A_x\subset \Bispan(x,a_i)\cap\Bispan(x,a_j)=\ell_k(x)\cup\ell_k(a_i)\cup\ell_k(a_j)\cup\ell_k(x_{ij}),
	$$
	where $\{1,2,3\}\setminus\{i,j\}=\{k\}$, and 
	$x_{ij}\in\ZZ_M$ is the unique point such that $(x_{ij}-a_i,M)=M/p_j$ and $(x_{ij}-a_j,M)=M/p_i$.
	\item 
	Suppose that $(x-a,M)=M/p_ip_j$ for some $a\in A$ and $i,j\in\{1,2,3\}$ distinct. Then
	$$
	A_x\subset \Pi(x,p_i^{n_i}) \cup\Pi(a,p_i^{n_i})\cup  \Pi(x,p_i^{n_j}) \cup\Pi(a,p_i^{n_j}).
	$$
	\item We leave it as an easy exercise for the reader to verify that if there are $a,a',a''\in A$ such that
	$(z-z',M)=M/p_ip_j$ for all pairs of distinct elements $z,z'\in\{x,a,a',a''\}$, then
	$$
	A_x\subset \Pi(x,p_i^{n_i}) \cup  \Pi(x,p_i^{n_j}) .
	$$
	
	\item Suppose that $x\in \ZZ_M\setminus A$ and $y\in \ZZ_M\setminus B$ with
	\begin{equation}\label{sat-layers-0}
	(x-a,M)=(y-b,M)=M/p_i \hbox{ for some } a\in A,b\in B.
	\end{equation}
	We claim that
	\begin{equation}\label{sat-layers}
		A_{x,y}\subset\Pi(x,p_i^{n_i-1}),\ \ B_{y,x}\subset\Pi(y,p_i^{n_i-1}).
	\end{equation}
	One way to prove this is as follows. Let $a\in A$ and $b\in B$ be as in (\ref{sat-layers-0}). As in the first example above, we have 
	$A_{x,b}\subset \Bispan(x,a)\subset \Pi(x,p_i^{n_i-1})$. Hence $B_{b,x} \subset \Pi(b,p_i^{n_i-1})=\Pi(y,p_i^{n_i-1})$.
	Applying (\ref{setplusspan}) to $B$, and using that $\Bispan(y,b)\subset \Pi(y,p_i^{n_i-1})$, we get that
	$B_{y,x} \subset \Pi(y,p_i^{n_i-1})$ as claimed. This also implies the first half of (\ref{sat-layers}).
	Alternatively, (\ref{sat-layers}) can also be deduced from Lemma \ref{triangles}.
	
\end{itemize}

Saturating sets are very useful in identifying configurations that {\em cannot} occur in tiling complements. For example, we have the following easy but important lemma.

\begin{lemma}[\bf No missing joints]\label{smallcube}
Let $A\oplus B=\ZZ_M$, where 
$M=p_1^{n_1}\dots p_K^{n_K}$. Suppose that
\begin{equation}\label{notopdivB}
\{D(M)|m|M\}\cap \Div(B)=\emptyset,
\end{equation}
and that for some $x\in\ZZ_M$ there exist $a_1,\dots, a_K\in A$ such that 
\begin{equation}\label{joint-e}
(x-a_i,M)=M/p_i \ \ \forall i\in\{1,\dots,K\}.
\end{equation}
Then $x\in A$.
\end{lemma}

\begin{proof}
Suppose that $x\not\in A$, and let $\Delta$ be the $M$-cuboid with vertices $x,a_1,\dots,a_K$. By (\ref{joint-e}) and (\ref{bispan}), the saturating set $A_x$ is contained in the vertex set of $\Delta$. But that is impossible by (\ref{notopdivB}).
\end{proof}

As an application, we prove the following restriction on fibered grids that can be a part of a tiling set. 

\begin{proposition}\label{prop-twodirections}
Let $M=p_1^{n_1}p_2^{n_3}p_2^{n_3}$. Assume that $A\oplus B=\ZZ_M$ is a tiling,
and that there exists a $D(M)$-grid $\Lambda$ such that $A\cap \Lambda$ is a nonempty union of disjoint $M$-fibers.
Then there is a subset $\{\nu_1,\nu_2\}\subset\{1,2,3\}$ of cardinality 2 such that $A\cap\Lambda$ is a union of disjoint $M$-fibers in the $p_{\nu_1}$ and $p_{\nu_2}$ directions.
\end{proposition}

\begin{proof}
Fix $A$ and $\Lambda$ as in the statement of the proposition.
We will say that $\kappa:A\cap\Lambda\to\{1,2,3\}$ is an {\em assignment function} if $A\cap\Lambda$ can be written as
$$
A\cap\Lambda=\bigcup_{a\in A\cap\Lambda} (a*F_{\kappa(a)}),
$$
where for any $a,a'\in A\cap\Lambda$, the fibers $a*F_{\kappa(a)}$ and $a'*F_{\kappa(a')}$ are either identical or disjoint.
Thus, if $a'\in a*F_{\kappa(a)}$, then $\kappa(a')=\kappa(a)$.
 Note that $\kappa$ is not necessarily unique, since there exist sets that can be split into nonintersecting fibers in more than one way. We will use $\Xi$ to denote the set of all assignment functions for $A\cap\Lambda$.

It suffices to prove that any assignment function $\kappa\in\Xi$ may take at most two values. To prove this, assume for contradiction that there exists $\kappa\in\Xi$ such that $\kappa(a_1)=1$, $\kappa(a_2)=2$, $\kappa(a_3)=3$ for some 
$a_1,a_2,a_3\in A\cap\Lambda$. Then the fibers $a_1*F_1$, $a_2*F_2$, $a_3*F_3$ are contained in $A$ and pairwise disjoint. 

Let $x\in\Lambda$ be the point such that
$$
\Pi(a_1,p_2^{n_2})\cap \Pi(a_2,p_3^{n_3})\cap \Pi(a_3,p_1^{n_1}) = \{x\}.
$$
Then there are points $a'_1\in a_1*F_1$, $a'_2\in a_2*F_2$, $a'_3\in a_3*F_3$ such that
$$
(x-a'_1,M)=M/p_3,\ (x-a'_2,M)=M/p_1,\ (x-a'_3,M)=M/p_2.
$$
Moreover, $\{D(M)|m|M\}\subset\Div(A\cap\Lambda)$, hence (\ref{notopdivB}) holds. By Lemma \ref{smallcube}, we must have $x\in A$. However, there is no permitted value for $\kappa(x)$, since $x*F_1$ intersects $a_2*F_2$, 
$x*F_2$ intersects $a_3*F_3$, and $x*F_3$ intersects $a_1*F_1$. This contradicts the definition of $\kappa$.
\end{proof}

For example, under the assumptions of Proposition \ref{prop-twodirections}, if $A\oplus B=\ZZ_M$ is a tiling, then $A\cap\Lambda$ cannot consist of three nonintersecting fibers in different directions.

\begin{figure}[h]
\captionsetup{justification=centering}
\includegraphics[scale=0.7]{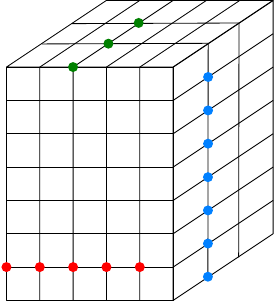}
\caption{A $D(M)$ grid with disjoint M-fibers in all 3 directions}
\end{figure}

\begin{remark}
Suppose that $\Phi_M|A$, where $M=p_1^{n_1}p_2^{n_2}p_3^{n_3}$. Let $\Lambda$ be a $D(M)$-grid such that $A\cap\Lambda\neq\emptyset$. As discussed in Section \ref{classic-cuboids}, $A\cap\Lambda(X)$ can be written as 
\begin{equation}\label{twooutofthree}
(A\cap\Lambda)(X)=\sum_{\nu\in\{1,2,3\}} Q_\nu(X) F_\nu(X),
\end{equation}
where $Q_1,Q_2,Q_3$ are polynomials with integer coefficients depending on both $A$ and $\Lambda$. If, in addition, $Q_1,Q_2,Q_3$ are polynomials with {\em nonnegative} coefficients, then $A\cap\Lambda$ is a nonempty union of disjoint $M$-fibers. By Proposition \ref{prop-twodirections}, if $A\oplus B=\ZZ_M$ is a tiling, then $A\cap\Lambda$ can be written in the form (\ref{twooutofthree}) with at least one of $Q_1,Q_2,Q_3$ equal to $0$.
\end{remark}

It is likely that some consistency conditions of this type occur more broadly in tiling sets. For example, in part (IIa) of \cite[Theorem 9.1]{LaLo2}, we prove a much more difficult and technical result of this type on a lower scale.




\section{Fibers and cofibers}\label{sec-fibers}


\subsection{Fibers and fiber chains}

\begin{definition}\label{standard-fibers}
Let $N|M$, and assume that $p_i^\delta\mid N$ for some $\delta\geq 1$. Define
\begin{equation}\label{sf-def2}
\Psi_{N/p_i^\delta}(X):= \Phi_{p_i}(X^{N/p_i^\delta}) = 1+X^{N/p_i^{\delta}}+ X^{2N/p_i^{\delta}} + \dots + X^{(p_i-1)  N/p_i^{\delta}}
\end{equation}
\end{definition}

This is the same notation as in (\ref{sf-def}), but here we are using it for a different purpose. Specifically, we will use polynomials of the form (\ref{sf-def2}) as building blocks for multiscale fibers and fiber chains below.
While $\Psi_{N/p_i^\delta}$ depends on both ${N/p_i^\delta}$ and $p_i$, both numbers will always be clear from the context. We will also use that
\begin{equation}\label{sf-def-divisors}
\Psi_{N/p_i^\delta}(X)= \frac{ X^{N/p_i^{\delta-1}}-1}{X^{N/p_i^{\delta}}-1}
=\prod_{s|M:\ s\neq 1, \ p_i^{\nu-\delta+1}\parallel s} \Phi_s(X),
\end{equation}
where $\nu\geq 1$ is the exponent such that $p_i^\nu\parallel N$.

\begin{definition}\label{fibers}
{\bf (Fibers and fiber chains)} Let $N|M$, and assume that $p_i| N$.

\smallskip
(i) 
A set $F_0\subset \ZZ_M$ is {\em an $N$-fiber in the $p_i$ direction} if $F_0$ mod $N$ has the mask polynomial
\begin{equation}\label{fib-e01}
F_0(X)\equiv  c X^a \Psi_{N/p_i}(X)    \mod X^N-1,
\end{equation}
with fixed $c\in\NN$ and $a\in\ZZ_N$. Equivalently, $F_0$ mod $N$ is a multiset in $\ZZ_N$ with weights
$$
w_{F_0}^N(x) = \begin{cases}
c & \hbox{ if }x\in \{ a, a+N/p_i, a+2N/p_i,\dots,a+(p_i-1) N/p_i\} , \\
0 & \hbox{ otherwise.}\\
\end{cases}
$$
If a fiber $F_0$ has the form (\ref{fib-e01}), we will refer to $c$ as its {\em multiplicity}, and will say that the fiber 
is {\em rooted at} $a$ or {\em passes through} $a$.

\smallskip
(ii) A set $A\subset \ZZ_M$ is {\em $N$-fibered in the $p_i$ direction} if it can be written as a union of disjoint 
$N$-fibers in the $p_i$ direction, all with the same multiplicity.

\smallskip
(iii) Let $\mathcal{P}\subset\{1,2, \dots,n_i\}$ be non-empty, where $p_i^{n_i}\parallel M$.
A set $F\subset \ZZ_M$ is {\em a $\mathcal{P}$-fiber chain in the $p_i$ direction} if 
$|F|= p_i^{|\mathcal{P}|}$ and $F$ is $N$-fibered in the $p_i$ direction for each $N=M/p_i^{\alpha-1}$, where
$\alpha\in\mathcal{P}$. We will also use the convention that if $\calp=\emptyset$, then a $\mathcal{P}$-fiber chain in any direction is any singleton set $\{x\}$ with $x\in\ZZ_M$.

\smallskip
(iv) A set $A\subset \ZZ_M$ is {\em $\mathcal{P}$-fibered in the $p_i$ direction} if it can be written as a union of disjoint 
$\mathcal{P}$-fiber chains in the $p_i$ direction.

\end{definition}

We list a few examples. Observe that, while we will not use $\calp$-fiber chains with multiplicities greater that 1, Definition \ref{fibers} (iii) does require the concept of $N$-fibers with multiplicity.

\begin{itemize}
\item A $\{1\}$-fiber chain in the $p_i$ direction  is simply an $M$-fiber in that direction, and an $\{1\}$-fibered set in the $p_i$ direction with multiplicity 1 is $M$-fibered in that direction, as defined in Section \ref{gridsetal}. 

\item A $\{2\}$-fiber chain in the $p_i$ direction is a set $F\subset\ZZ_M$ such that for some $a\in \ZZ_M$ we have
$F(x)\equiv X^a(1+X^{M/p_i^2}+X^{2M/p_i^2}+ \dots +X^{(p_i-1)M/p_i^2})$ mod $(X^{M/p_i}-1)$.
Note that $|F|=p_i$.

\item A $\{1,2\}$-fiber chain in the $p_i$ direction is a set $F\subset\ZZ_M$ such that for some $a\in \ZZ_M$ we have
$F(x)\equiv X^a(1+X^{M/p_i^2}+X^{2M/p_i^2}+ \dots +X^{(p_i^2-1)M/p_i^2})$ mod $(X^{M}-1)$.
Note that $|F|=p_i^2$, $F$ is $M$-fibered in the $p_i$ direction with multiplicity 1, and $M/p_i$-fibered in the $p_i$ direction with multiplicity $p_i$.
\end{itemize}

\begin{lemma}\label{fibered-properties}
{\bf (Properties of fibered sets)} 
Assume that $A\subset \ZZ_M$ is $\mathcal{P}$-fibered in the $p_i$ direction for some $\mathcal{P}\subset\{1,2,\dots,n_i\}$.
Then the following hold.

\smallskip

(i) We have 
\begin{equation}\label{fiffib-a}
\prod_{\alpha\in \mathcal{P}} \Psi_{M/p_i^{\alpha}}(X)|A(X).
\end{equation}
In particular, $\Phi_s(X)|A(X)$ for all $s|M$, $s\neq 1$, such that $s=p_i^{n_i-\alpha+1}s'$, where $\alpha\in \mathcal{P}$ and $(s',p_i)=1$.

\smallskip

(ii) We have $p_i^{|\mathcal{P}|} \mid |A|$. In particular, a $\mathcal{P}$-fiber chain $F$ as in Definition \ref{fibers} (iii)
is a minimal set that is $\mathcal{P}$-fibered in the $p_i$ direction.

\smallskip
(iii) $\{M/p_i^{\alpha}:\ \alpha \in \mathcal{P}\}\subset\Div(A)$.

\smallskip

(iv)  Let $F\subset \ZZ_M$ be a $\mathcal{P}$-fiber chain with multiplicity 1 in the $p_i$ direction. Translating $F$ if necessary,
we may assume that $0\in F$. Let $\gamma =\max \mathcal{P}$.
Then $F\subset (M/p_i^\gamma)\ZZ_M\subset \ell_i(0)$, and 
$F$ tiles $(M/p_i^\gamma)\ZZ_M\simeq \ZZ_{p_i^\gamma}$ with the standard tiling complement $G$, where
$$G(X)=\prod_{\tau: \ 1\leq\tau<\gamma,\  \tau\notin\mathcal{P}} \Psi_{M/p_i^{\tau}}(X).$$
(We use the convention that an empty product is equal to 1.)

\end{lemma}

\begin{proof}
Part (i) follows directly from the definition, and (ii) follows from (i) since $\Phi_{M/p_i^\alpha}(1)=p_i$. For (iii), let $\alpha\in\calp$, and let $a,a'\in A$ be elements that belong to the same $M/p_i^{\alpha-1}$-fiber in the $p_i$ direction, but not to the same $M/p_i^\beta$-fiber in the $p_i$ direction for any $\beta<\alpha-1$. Then $(a-a',M)=M/p_i^\alpha$, as claimed. 

We now prove (iv). Assume that $0\in F$. Since $\Psi_{M/p_i^\alpha}(X)=\Phi_{p_i^{\gamma-\alpha+1}}(X^{M/p_i^\gamma})$ for $\alpha\leq\gamma$, 
by (i) we have $F(X)\equiv Q(X)\Psi(X^{M/p_i^\gamma})$ mod $(X^M-1)$, where
$$\Psi(X)=\prod_{\alpha\in \mathcal{P}} \Phi_{p_i^{\gamma-\alpha+1}}(X).
$$
Since $\Psi(1)=\prod_{\alpha\in \mathcal{P}}p_i = |F|$, we have $Q(1)=1$. Splitting up the weighted multiset corresponding to $Q(X)$ into residue classes mod $M/p_i^\gamma$, and using that $F$ is a set, we see that 
$Q\in\calm((M/p_i^\gamma)\ZZ_M)$. Hence $F\subset (M/p_i^\gamma)\ZZ_M$.

Let $F'=\{\frac{x}{M/p_i^\gamma}:\ x\in F\}$, so that  
$F' \subset \ZZ_{p_i^\gamma}$ and $\Psi(X)\mid F'(X),$
with $\Psi(1)= |p_i|^{|\calp|} = |F'|$. This is the (T1) tiling condition for $F'$. Hence $F'$ tiles $\ZZ_{p_i^\gamma}$ with the standard tiling complement (see Remark \ref{1dim-standard}). Rescaling back to $F\subset\ZZ_M$, we get (iv).
\end{proof}


\subsection{Cofibers and cofibered structures}

Given a tiling $A\oplus B=\ZZ_M$, we will be interested in the occurrences of ``complementary" fiber chains in $A$ and $B$, in the following sense. 

\begin{definition}[\bf Cofibers]\label{single-cofibers} 
Let $A, B\subset \ZZ_M$, and fix $1\leq\gamma \leq n_i$. 
Let $\mathcal{P}_A, \mathcal{P}_B$ be two disjoint sets
such that 
\begin{equation}\label{cofib-e10}
 \mathcal{P}_A\cup\mathcal{P}_B=\{1,2,\ldots,\gamma\}.
\end{equation}
We say that $F\subset A,G\subset B$ are {\em $(\mathcal{P}_A, \mathcal{P}_B)$-cofibers} in the $p_i$ direction
if:

\begin{itemize}

\item $F$ is a $\mathcal{P}_A$-fiber chain in the $p_i$ direction,

\item $G$ is a $\mathcal{P}_B$-fiber chain in the $p_i$ direction.
\end{itemize}
We will also refer to $(F,G)$ as a {\em $(\mathcal{P}_A, \mathcal{P}_B)$-cofiber pair.}
\end{definition}

Note that if $\gamma=1$, then one of the sets $\calp_A$ and $\calp_B$ must be empty. If $\gamma=1$ and $\calp_A=\emptyset$, then $F$ is a singleton and $G$ is an $M$-fiber in the $p_i$ direction.

Our goal will be to find global cofibered structures as described below. If $A\oplus B=\ZZ_M$ is a tiling pair, having a cofibered structure will often allow us to reduce proving (T2) for $(A,B)$ to proving it to an equivalent but simpler tiling pair. In order to allow for intermediate steps involving sets that are only partially fibered, we state the definition below for arbitrary sets $A,B\subset\ZZ_M$.

\begin{definition}[\bf Cofibered structure and cofibered sets]\label{cofibers} 
Let $A, B\subset \ZZ_M$, and fix $1\leq\gamma \leq n_i$. 
Let $\mathcal{P}_A, \mathcal{P}_B$ be two disjoint sets obeying (\ref{cofib-e10}).

(i) We say that the pair $(A,B)$ has a {\em $(\mathcal{P}_A, \mathcal{P}_B)$-cofibered structure} in the $p_i$ direction if:

\begin{itemize}
\item $B$ is  $\mathcal{P}_B$-fibered in the $p_i$ direction,

\item $A$ contains at least one ``complementary" $\mathcal{P}_A$-fiber chain $F\subset A$
 in the $p_i$ direction, which we will call a  {\em cofiber} for this structure. We will say that $F$ is rooted at $a\in A$
if $a\in F$.
\end{itemize}

(ii) 
We say that the pair $(A,B)$ is {\em $(\mathcal{P}_A, \mathcal{P}_B)$-cofibered} in the $p_i$ direction if:

\begin{itemize}
\item $A$ is  $\mathcal{P}_A$-fibered in the $p_i$ direction,

\item $B$ is  $\mathcal{P}_B$-fibered in the $p_i$ direction.
\end{itemize}

\end{definition}

We emphasize that part (i) of the definition is {\em not} symmetric with respect to $A$ and $B$. Our convention is that the
{\em second} set in the pair must be fibered in its entirety.
While a cofibered structure may have more than one cofiber in $A$, we do not require that
the entire pair $(A,B)$ be cofibered. We will refer to the number $\gamma$ in Definitions \ref{single-cofibers} and \ref{cofibers} as the {\em depth} of, respectively, the cofiber pair or the cofibered structure.

If $A$ and $B$ satisfy the condition of Definition \ref{cofibers} (i), then
by Lemma \ref{fibered-properties},
\begin{equation}\label{cofib-e12}
\{M/p_i^{\alpha}:\ \alpha \in \mathcal{P}_A\}\subseteq\Div(A),\ \ 
\{M/p_i^{\beta}:\ \beta \in \mathcal{P}_B\}\subseteq  \Div(B),
\end{equation}
\begin{equation}\label{cofiffib-b}
\prod_{\beta\in\mathcal{P}_B} \Psi_{M/p_i^{\beta}}(X)|B(X),
\end{equation}
and if a cofiber $F$ is rooted at some $a\in A$, then
\begin{equation}\label{cofiffib-a}
X^a\prod_{\alpha\in \mathcal{P}_A} \Psi_{M/p_i^{\alpha}}(X) \Big| F(X).
\end{equation}

\begin{remark}\label{fibered-critters} 
Assume that $A, B\subset \ZZ_M$ satisfy $\Div(A)\cap\Div(B)=\{M\}$, and fix $1\leq\gamma \leq n_i$. 
Let $\mathcal{P}_A, \mathcal{P}_B$ be two disjoint sets obeying (\ref{cofib-e10}). 
Assume that
\begin{equation}\label{fibered-critters-eq} 
\{M/p_i^{\alpha}:\ \alpha \in \mathcal{P}_A\}\cap\Div(B)=\emptyset.
\end{equation}
(In particular, if $A\oplus B=\ZZ_M$ and $A$ contains a $\calp_A$-fiber chain in the $p_i$ direction, then (\ref{fibered-critters-eq}) holds by
Lemma \ref{fibered-properties} (iii) and divisor exclusion.) 
Then, in order to prove that $B$ is $\mathcal{P}_B$-fibered in the $p_i$ direction, it suffices to verify 
that every $b\in B$ belongs to a $\mathcal{P}_B$-fiber chain $F(b)$ in the $p_i$ direction. 
Indeed, by Lemma \ref{fibered-properties} (iv),
every $F(b)$ is a maximal subset of $b*(M/p_i^\gamma)\ZZ_M$ such that $\Div(F(b))\cap \{M/p_i^{\alpha}:\ \alpha \in \mathcal{P}_A\}=\emptyset$. Hence, under the above assumptions, any two fiber chains $F(b)$ and $F(b')$ with $b,b'\in B$ must be either identical or disjoint.
\end{remark}

\subsection{Fiber shifting}

Cofibered structures are important for two reasons. On one hand, they arise naturally from 1-dimensional
saturating spaces (see Lemma \ref{1dim_sat} below). On the other hand, with a cofibered structure in place, 
Lemma \ref{fibershift} allows us to shift the cofibers in $A$ as indicated while maintaining both
the tiling property and the (T2) status of $A$. Applying such shifts repeatedly, we are able to reduce many cases
to simpler tilings where (T2) is easy to verify.

\begin{lemma}[\bf Fiber-Shifting Lemma]\label{fibershift} 
Let $A\oplus B=\ZZ_M$. Assume that the pair $(A,B)$ has a $(\mathcal{P}_A, \mathcal{P}_B)$-cofibered structure,
with a cofiber $F\subset A$. Let $A'$ be the set obtained from $A$ by shifting $F$ by 
$M/p_i^{\beta}$ for any $\beta\in\mathcal{P}_B$. 
Then $A'\oplus B=\ZZ_M$, and $A$ is T2-equivalent to $A'$.
\end{lemma}

\begin{proof} 
We have
$$
A'(X)=A(X)+(X^{kM/p_i^{\beta}}-1)F(X)
$$
for some $k$ with $(k,p_i)=1$.

We must prove that $\Phi_s(X)|A'(X)B(X)$ for all $s|M, s\neq 1$. Fix such $s$, and write it as $s=p_i^{n_i-\gamma}s'$, where $(s,p_i)=1$. 
Consider three cases.
\begin{itemize}

\item If $\gamma\geq  \beta$, then $\Phi_s(X)\mid (X^{kM/p_i^{\beta}}-1)$, therefore it divides $A$ if and only if it divides $A'$.

\item If $\gamma < \beta$ and $\gamma\in \mathcal{P}_B$, then $\Phi_s(X)\mid \Psi_{M/p_i^{\gamma}}(X) \mid B(X)$.

\item If $\gamma < \beta$ and $\gamma\in \mathcal{P}_A$, then $\Phi_s(X)\mid \Psi_{M/p_i^{\gamma}}(X) \mid F(X)$,
therefore $\Phi_s$ divides $A$ if and only if it divides $A'$.
\end{itemize}
This implies the first part of the lemma.

Suppose furthermore that $\Phi_s$ is a (T2) cyclotomic polynomial of $A$, in the sense that $s=s_1\dots s_\tau$, where $s_1,\dots,s_\tau$ are powers of distinct primes such that $\Phi_{s_1}\dots\Phi_{s_\tau}|A$. In particular, we must have
$\Phi_{p_i^{n_i-\gamma}}\mid A$, and therefore $\Phi_{p_i^{n_i-\gamma}}\nmid B$. By the above analysis applied to $p_i^{n_i-\gamma}$ instead of $s$, we must have 
either $\gamma \geq \beta$ or $\gamma\in \mathcal{P}_A$. In both cases, we get that $\Phi_s$ divides $A$ if and only if it divides $A'$, so that the (T2) property is preserved when we pass from $A$ to $A'$.
\end{proof}

\subsection{Fibers and 1-dimensional saturating spaces} We now prove that 1-dimensional saturating sets imply cofibered structures.

\begin{lemma}\label{onedivisor}
Assume that $A\oplus B=\ZZ_M$, and let $x,y\in\ZZ_M$.

\medskip

(i) Let $1\leq \alpha,\alpha'\leq n_i$ with $\alpha\neq \alpha'$. Then
$$
\bbA_{M/p_i^{\alpha}} [x] \,\bbB_{M/p_i^{\alpha}} [y] \, \bbA_{M/p_i^{\alpha'}} [x]
 \, \bbB_{M/p_i^{\alpha'}}  [y] =0.
$$
In particular, if $A_{x,y}\subset \ell_i(x)$, then the product $\langle \bbA[x],\bbB[y]\rangle$ is saturated by a single divisor.

\medskip

(ii) Suppose that $A_x\subset \ell_i(x)$. Then there exists an $\alpha$ with $0\leq\alpha\leq n_i$
such that 
$$\bbA_{M/p_i^{\alpha}} [x] \,\bbB_{M/p_i^{\alpha}} [b] = \phi(p_i^\alpha)\hbox{ for all }b\in B.$$

\end{lemma}

\begin{proof}
Part (i) is a special case of Lemma \ref{triangles}, and part (ii) follows from (\ref{bispan}).
\end{proof}

\begin{definition}
Let $\mathcal{P}\subset\{1,2,\dots,n_i\}$,
and let $F\subset \ZZ_M$ be a $\mathcal{P}$-fiber chain in the $p_i$ direction.

\smallskip

(i) An element $x\in \ZZ_M$ is at {\em distance} $m$ from $F$ if $m|M$ is the maximal divisor
such that $(z-x,M)=m$ for some $z\in F$. 

\smallskip

(ii) If $1\leq\delta\leq n_i$, we will write $\mathcal{P}[\delta]= \mathcal{P}\cap \{1,2,\dots,\delta\}$.

\end{definition}

If $x\in\ZZ_M$ and $F\subset \ZZ_M$ is a $\mathcal{P}$-fiber chain in the $p_i$ direction,
then for all $z\in F$ we have $(z-x,M)=m'p_i^{\alpha(z)}$, where $m'|(M/p_i^{n_i})$
is the same for all $z\in F$. In particular, the distance from $x$ to $F$ is well defined and is equal to
$m'p_i^{\max_{z\in F}\alpha(z)}$.

\begin{lemma}[\bf The structure of 1-dimensional saturating spaces]\label{1dim_sat}
Assume that $A\oplus B =\ZZ_M$ is a tiling.

\smallskip
(i) Suppose that $ x, y\in\ZZ_M $ satisfy $x\notin A$ and
\begin{equation}\label{fib-e7}
\mathbb{A}^M_{M/p_i^\gamma}[x]\mathbb{B}^M_{M/p_i^\gamma}[y]=\phi(p_i^\gamma)
\end{equation}
for some $ 0<\gamma\leq n_i $. 
Then there exist two disjoint sets 
$\mathcal{P}_A,\mathcal{P}_B $ with
\begin{equation}\label{fib-e10}
\mathcal{P}_A\cup\mathcal{P}_B = \{1,\dots,\gamma-1\},
\end{equation}
\begin{equation}\label{fib-e12}
\{M/p_i^{\alpha}:\ \alpha \in \mathcal{P}_A\}\subseteq\Div(A),\ \ 
\{M/p_i^{\beta}:\ \beta \in \mathcal{P}_B\}\subseteq  \Div(B),
\end{equation}
such that the following holds. Let 
$A_0\subset A_{x,y}$ be a maximal subset such that for all $a,a'\in A_0$ with $a\neq a'$ we have $(a-a',M)=M/p_i^{\gamma}$, and let $B_0$ be a similar subset of $B_{y,x}$.
Then one of the sets $A_0$ and $B_0$ has cardinality 1, the other has cardinality $p_i-1$, and furthermore
\begin{equation}\label{fib-e100}
A_{x,y}=\bigcup_{a\in A_0}F(a),\ \ B_{y,x}= \bigcup_{b\in B_0}G(b),
\end{equation}
where $F(a)$ is a $\mathcal{P}_A$-fiber chain in the $p_i$ direction rooted at $a$, and
$G(b)$ is a $\mathcal{P}_B$-fiber chain in the $p_i$ direction rooted at $b$.

\smallskip
(ii) Suppose that $x\in\ZZ_M\setminus A$ and $A_x\subset\ell_i(x)$, with 
\begin{equation}\label{fib-e7b}
\mathbb{A}^M_{M/p_i^\gamma}[x]\mathbb{B}^M_{M/p_i^\gamma}[b]=\phi(p_i^\gamma)
\hbox{ for all }b\in B,
\end{equation}
where $ 0<\gamma\leq n_i $ (as follows from Lemma \ref{onedivisor} (ii)).
Then the pair $(A,B)$ has a $(\calp_A,\calp_B\cup\{\gamma\})$-cofibered structure, with 
$A_{x}$ as a $\calp_A$-cofiber at a distance $M/p_i^\gamma$ from $x$.
\end{lemma}

\begin{proof}
We first prove (i). Define $A_0$ and $B_0$ as above. Since $A_0\subset x*(M/p_i^{\gamma})\ZZ_M$ and each element of $A_0\cup\{x\}$ is contained in a different residue class mod $M/p_i^{\gamma}$, we have 
$|A_0|\leq p_i-1$, and similarly for $B_0$.
By divisor exclusion, at most one of these sets has cardinality greater than 1.

Next, let $A_1\subset x*(M/p_i^{\gamma})\ZZ_M$ be a maximal subset of $A_{x,y}$ such that
$$
\forall a,a'\in A_1 \hbox{ with } a\neq a', \hbox{ we have }(a-a',M)=M/p_i^{\gamma-1},
$$
and define $B_1$ similarly. Then $|A_1|\leq p_i|A_0|$, since for each $a\in A_1$ there must be a ``parent" $a_0\in A_0$ with 
$M/p_i^{\gamma-1}\mid a-a_0$, and each $a_0$ can have at most $p_i$ such ``children" $a\in A_1$ (we allow $a=a_0$, so that $A_0\subset A_1$). Similarly, $|B_1|\leq p_i|B_0|$. Moreover, if $|A_1|>|A_0|$, then we must have $M/p_i^{\gamma-1}\in\Div(A)$, and similarly for $B$, so that at least one of $|A_1|=|A_0|$ and $|B_1|=|B_0|$ must hold.
If $|A_1|>|A_0|$, we place $\gamma-1$ in $\calp_A$, otherwise we place it in $\calp_B$.

We continue by induction, constructing a sequence of sets 
$$
A_0\subset A_1\subset A_2\subset\dots\subset A_{\gamma-1}=A_{x,y},\ \ 
B_0\subset B_1\subset B_2\subset\dots\subset B_{\gamma-1}=B_{y,x},
$$
and two disjoint sets $\mathcal{P}_A,\mathcal{P}_B $ obeying (\ref{fib-e10}), 
so that for each $l=1,2,\dots, \gamma-1$:
\begin{itemize}
\item if $l\in\calp_A$, then $|A_{\gamma-l+1}|\leq p_i|A_{\gamma-l}|$, $|B_{\gamma-l+1}|=|B_{\gamma-l}|$,
and $M/p_i^{\gamma-l}\not\in\Div(B)$,
\item if $l\in\calp_B$, then the same holds with $A$ and $B$ interchanged.
\end{itemize}
 It follows that 
$$
\mathbb{A}^M_{M/p_i^\gamma}[x]\mathbb{B}^M_{M/p_i^\gamma}[y]
\leq |A_0|\, |B_0|  \, p_i^{|\calp_A|} p_i^{|\calp_B|} 
\leq (p_i-1) p_i^{\gamma -1} = \phi(p_i^\gamma).
$$
Furthermore, for the equality to hold, one of the sets $A_0,B_0$ must have cardinality $p_i-1$, for each $a\in A_0$ the
set $F(a):=\{a\in A_{x,y}: M/p_i^{\gamma-1}\mid a-a_0\}$ 
must be a full $\mathcal{P}_A$-fiber chain in the $p_i$ direction rooted at $a$, and a similar statement must hold for $B$. This yields the structure described in part (i).

For part (ii), assume that (\ref{fib-e7b}) holds, and let $B_0(b)$ be the set from (\ref{fib-e12}) with $y=b$ for each $b\in B$.
Since $M/p_i^\gamma\in\Div(B)$, we must have $|A_0|=1$ and $|B_0(b)|=p_i-1$,
Fix $b\in B$, so that 
$$
B_{b,x}= \bigcup_{b'\in B_0(b)}G(b').
$$
Let $b'\in B_0(b)$, and apply part (i) of the lemma with $y=b'$. Since $b\in B_{b',x}$, there is a $\calp_B$-fiber chain $G(b)\subset B$ rooted at $b$, so that 
$$
B_{b',x}= G(b)\cup \bigcup_{b''\in B_0(b),b''\neq b'}G(b'').
$$
Thus $\bigcup_{b''\in B_0(b)\cup\{b\}}G(b'')$ is a $(\calp_B\cup\{\gamma\})$-fiber chain in $B$, rooted at $b$.
Applying this argument to all $b\in B$, and using Remark \ref{fibered-critters},
we get the cofibered structure as indicated.
\end{proof}

The following special case will be used frequently in \cite{LaLo2}.

\begin{corollary}\label{1dim_sat-cor}
Assume that $A\oplus B =\ZZ_M$ is a tiling.
Suppose that $x\in\ZZ_M\setminus A$, $b\in B$, $M/p_i\in\Div(A)$, and
\begin{equation}\label{fib-e300}
\mathbb{A}^M_{M/p_i^2}[x]\mathbb{B}^M_{M/p_i^2}[b]=\phi(p_i^2).
\end{equation}
Then there exists a $(\{1\},\{2\})$-cofiber pair $(F,G)$ such that $F\subset A$ is at distance $M/p_i^2$ from $x$, $G\subset B$ is rooted at $b$, and 
\begin{equation}\label{fib-e301}
\mathbb{A}^M_{M/p_i^2}[x| F]\mathbb{B}^M_{M/p_i^2}[b | G]=\phi(p_i^2).
\end{equation}
In particular, if $M/p_i\in\Div(A)$ and $A_x\subset \ell_i(x)$ with $M/p_i^2$ as the contributing divisor (cf. Lemma \ref{onedivisor} (ii)),
then the pair
$(A,B)$ has a $(\{1\},\{2\})$-cofibered structure.
\end{corollary}

For simplicity, when $M$ is fixed, we will write ``$(1,2)$-cofiber pair" instead of ``$(\{1\},\{2\})$-cofiber pair", and similarly for cofibered structures.


\subsection{Examples and applications}\label{shiftexamples}

Let $M=p_1^{n_1}\dots p_K^{n_K}$ with $K\geq 3$ and $p_1,\dots,p_K\geq 3$. Assume that $A\oplus B=\ZZ_M$ and $|A|=p_1\dots p_K$. Let also $\Lambda$ be a fixed $D(M)$ grid, and assume that $0\in A\cap\Lambda$.

\medskip\noindent
{\bf Example 1.} 
By Lemma \ref{smallcube}, we cannot have $\Lambda\setminus A=\{x\}$ for a single point $x\in\Lambda$. Similarly, we cannot have $A\cap\Lambda=A_0$ if $A_0$ is obtained from $\Lambda$ by deleting a few more points in an ``unstructured" way so that the assumptions of Lemma \ref{smallcube} still apply.

Suppose, however, that $A_0=\Lambda \setminus (x*F_i)$ for some $x\in\Lambda$ and $i\in\{1,\dots,K\}$. Then Lemma \ref{smallcube} is no longer applicable, and indeed, it is possible to have $A\cap\Lambda=A_0$. However, as we now show, this determines the structure of the entire set $A$, and, in particular, both $A$ and $B$ satisfy (T2).

Indeed, we have $\bbA_{M/p_j}[x]\geq 2$ for all $j\neq i$. It follows by (\ref{bispan}) that $A_x\subset \ell_i(x)$. By 
Proposition \ref{1dim_sat}, the pair $(A,B)$ has a $(\mathcal{P}_A,\mathcal{P}_B)$-cofibered structure of depth $\gamma\geq 2$, with $1\in\calp_A$ since $M/p_i\in\Div(A)$. In particular, $A$ must contain an $M$-fiber in the $p_i$ direction at distance $M/p_i^\gamma$ from $x$. By Lemma \ref{fibershift}, we can shift that fiber to $x$, proving that $A$ is T2-equivalent to $\Lambda$. Thus $A^\flat=\Lambda$, and Corollary \ref{get-standard} implies 
(T2) for both $A$ and $B$.

We note that the same argument still applies if $A\cap\Lambda$ has several fibers missing (possibly in different directions). This is the case e.g., in Szab\'o-type examples in \cite{Sz}, \cite{LS}.

\medskip\noindent
{\bf Example 2.} We now consider a more difficult example where saturating sets are not as obvious. 
Let $M=p_i^{2}p_j^{2}p_k^{2}$ with $p_i,p_j,p_k\geq 3$, and assume that $|A|=p_ip_jp_k$ tiles $\ZZ_M$.
Suppose that there exists an element $x\in\Lambda \setminus A$ such that 
\begin{equation}\label{fullplanecond}
\bbA_{M/p_i}[x]=\phi(p_i), \bbA_{M/p_jp_k}[x]=\phi(p_jp_k)
\end{equation}
and $\bbA_m[x]=0$ for all $m\in \{D(M)|m|M\}\setminus \{M/p_i,M/p_jp_k\} $. In the terminology of \cite{LaLo2}, this is a {\em $p_i$-full plane structure}. We prove in \cite[Section 7]{LaLo2} that, for a broader class of tilings including this situation, 
we have $A^\flat=\Lambda$ and the tiling $A\oplus B=\ZZ_M$ is T2-equivalent to $\Lambda\oplus B=\ZZ_M$ via fiber shifts. By Corollary \ref{get-standard}, both $A$ and $B$ satisfy (T2). For expository purposes, we restrict our attention here to this specific structure.

\begin{figure}[h]
\includegraphics[scale=0.7]{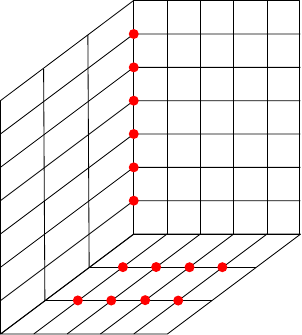}
\caption{A $p_i$-full plane structure on a $D(M)$ grid.}
\end{figure}

Consider the saturating set $A_x$, with $x$ as above. This time, geometric restrictions alone are not sufficient to confine $A_x$ to a single line through $x$. Nonetheless, with an additional argument we have the following lemma.

\begin{lemma}\label{fullplanesplitting}
Under the assumptions of Example 2, we have either $A_x\subset \ell_j(x)$ or $A_x\subset \ell_k(x)$. 
\end{lemma}

\begin{proof}
By (\ref{bispan}), we have
\begin{equation}\label{twotails}
A_x\subset \ell_j(x)\cup\ell_k(x).
\end{equation}
Let $b\in B$. Suppose that $A_{x,b}\cap \ell_j(x)$ is nonempty. Since $M/p_j\in\Div(A)$, we must have 
\begin{equation}\label{jline}
\bbA_{M/p_j^2}[x]\bbB_{M/p_j^2}[b]>0.
\end{equation}
If we also had $A_{x,b}\cap \ell_k(x)\neq\emptyset$ for the same $b$, that would imply that 
\begin{equation}\label{kline}
\bbA_{M/p_k^2}[x]\bbB_{M/p_k^2}[b]>0;
\end{equation}
however, having (\ref{jline}) and (\ref{kline}) at the same time would
contradict Lemma \ref{triangles}. Therefore we must have either $A_{x,b}\subset \ell_j(x)$ or $A_{x,b}\subset\ell_k(x)$. Notice that in the former case we have 
\begin{equation}\label{p_jcofiber}
\bbA_{M/p_j^2}[x]\bbB_{M/p_j^2}[b]=\phi(p_j^2).
\end{equation}
Since $M/p_j\in \Div(A)$, this can only happen if 
\begin{equation}\label{cofiberbreakdown}
\bbA_{M/p_j^2}[x]=p_j \text{ and } \bbB_{M/p_j^2}[b]=\phi(p_j).
\end{equation}
In other words, the pair $(A\cap\ell_j(x), B\cap\ell_j(b))$ contains a $(1,2)$-cofiber pair in the $p_j$ direction. 
If $A_{x,b}\subset\ell_k(x)$, (\ref{p_jcofiber}) and (\ref{cofiberbreakdown}) hold with $j$ replaced by $k$.

We claim that either $A_x\subset \ell_j(x)$ or $A_x\subset\ell_k(x)$. Indeed, assume for contradiction that there exist $b_j,b_k\in B$ such that $A_{x,b_j}\subset \ell_j(x)$ and $A_{x,b}\subset\ell_k(x)$. It follows from (\ref{cofiberbreakdown}) that 
\begin{align*}
|A\cap\Pi(x,p_i^{n_i})|&\geq \bbA_{M/p_j^2}[x]+\bbA_{M/p_k^2}[x]+\bbA_{M/p_jp_k}[x]	\\
&=p_j+p_k+(p_j-1)(p_k-1)\\
&=p_jp_k+1.
\end{align*}
This, however, contradicts Lemma \ref{planebound}. 
\end{proof}

Assume, without loss of generality, that $A_x\subset \ell_j(x)$. By Corollary \ref{1dim_sat-cor}, the pair
 $(A,B)$ has a $(1,2)$-cofibered structure in the $p_j$ direction, with a cofiber in $A$ at distance $M/p_j^2$ from $x$. By Lemma \ref{fibershift}, we may shift the cofiber to $x$. Let $A'$ be the set thus obtained, so that $A'\cap\Lambda$ contains all points of $A\cap\Lambda$ plus, additionally, the fiber $x*F_j\subset A'$. Moreover, $A'$ is T2-equivalent to $A$, and $A'\oplus B=\ZZ_M$. 

In this example, we do not get T2-equivalence to a standard set right away. Instead, the new set $A'$ contains a structure we call a {\em $p_j$-corner} \cite{LaLo2}, consisting of two nonintersecting $M$-fibers in the $p_i$ and $p_k$ directions in $\Lambda$. We then have to work further with that structure to prove that, ultimately, $A'$ (therefore $A$) is T2-equivalent to $\Lambda$.


\section{Conjectures and open questions}\label{conjectures-section}

\subsection{Tiling reductions}

We first consider the question of whether proving properties such as (T2), or, more generally, proving structure and classification results for tilings, could be accomplished by inductive arguments involving reduction to tilings of smaller groups. 

Let $A\oplus B=\ZZ_M$ be a tiling, and assume for convenience that $0\in A\cap B$.
If $M$ has at most two distinct prime factors, then Sands's theorem
\cite{Sands} states that at least one of $A$, $B$ must be contained in $p\ZZ_M$ for some prime $p|M$.  
Thus we can always use Theorem \ref{subgroup-reduction} to decompose such a tiling into tilings of residue classes, with at least one of the sets $A$ and $B$ tiling $p\ZZ_M$. This was the route taken in \cite{CM}.

Suppose now that $M=\prod_{i=1}^K p_i^{n_i}$, where $p_i$ are distinct primes and $K\geq 3$. 
Sands's theorem no longer holds in that setting, with counterexamples given by Szab\'o 
\cite{Sz} (see also \cite{LS}). 
However, it is conceivable that other inductive arguments, not based on Theorem \ref{subgroup-reduction}, may still apply. For example, the following question is open.

\begin{question}\label{q-slab}
Let $A\oplus B=\ZZ_M$ with $M=\prod_{i=1}^K p_i^{n_i}$. 

\smallskip
(i) {\em (Strong version)} Suppose that $\Phi_{p_i^{n_i}}|A$ for some $i\in\{1,2,\dots,K\}$. Is it always true that, in the notation of Theorem \ref{subtile}, 
we have $A'_{p_i}\oplus B=\ZZ_{M/p_i}$ for every translate $A'$ of $A$?

\smallskip
(ii) {\em (Weak version)} Must there always exist some $i\in\{1,2,\dots,K\}$ such that either $A'_{p_i}\oplus B=\ZZ_{M/p_i}$ for every translate $A'$ of $A$, or $A\oplus B'_{p_i}=\ZZ_{M/p_i}$ for every translate $B'$ of $B$?
\end{question}

We do not know of any counterexamples to this. Szab\'o's examples \cite{Sz} satisfy the conditions of Theorem \ref{subtile}, as do all tilings of period $M=p_1^2p_2^2p_3^2$, where $p_1,p_2,p_3$ are all odd \cite{LaLo2}. 

Assume that $\Phi_{p_i^{n_i}}|A$ for some $i\in\{1,2,\dots,K\}$.
By Proposition \ref{replacement}, the property (T2) for $B$ is equivalent to $A^\flat \oplus B=\ZZ_M$, where $A^\flat$ is the corresponding standard tiling complement. Heuristically, the slab reduction could be thought of as going part of the way in that direction, with the original tile $A$ replaced by a new tile $S$ which keeps some of the structure of $A$ but, additionally, is periodic in the $p_i$ direction. On the other hand, even if we assume {\em a priori} that both $A$ and $B$ satisfy (T2), this does not appear to imply the slab reduction in any obvious formal way. We do not know whether it is always possible to start with the original tiling and reach $A^\flat \oplus B=\ZZ_M$ via a sequence of slab reductions or other similar steps. While it does follow from \cite{LaLo2} that all tilings of odd period $M=p_1^2p_2^2p_3^2$ satisfy the conditions of Theorem \ref{subtile}, this is obtained {\em a posteriori} as a consequence of our classification of all such tilings, with (T2) and the classification results obtained by other means in some cases.

It is worthwhile to describe Szab\'o-type examples in more detail. (For the purpose of this paper, we use a modification of Szab\'o's original construction in \cite{Sz}, which was set in a different abelian group but was nonetheless based on the same idea. See also the examples in \cite{LS} and \cite{dutkay-kraus}.)
We start with the standard tiling $A^\flat \oplus B^\flat=\ZZ_M$, where $M=p_1^2p_2^2p_3^2$, $A^\flat$ is the standard tiling set with $\Phi_{p_i^2}|A$ for all $i\in\{1,2,3\}$, and $B^\flat$ is the standard tiling set with $\Phi_{p_i}|B$ for all $i\in\{1,2,3\}$. 
We then use fiber shifts (Lemma \ref{fibershift}) to modify $A^\flat$ so that for each $i$, one $M$-fiber in $A^\flat$ in the $p_i$ direction is shifted by a distance $M/p_i^2$. For $K=3$, the $M$-fibers in all 3 directions can be selected so that all three shifts can be performed independently without destroying the tiling property.  
This produces a new tiling $A \oplus B^\flat=\ZZ_M$ in which 
neither $A$ nor $B^\flat$ is contained in a proper subgroup of $\ZZ_M$.

Noting that the pair $(A,B^\flat)$ in the above construction has a $(1,2)$-cofibered structure in all three directions, one might ask whether one of $A$ and $B$ must in fact be contained in a proper subgroup if no such obstructions are present. This motivates the following question.

\begin{question}\label{q-obstruction}
Let $A\oplus B=\ZZ_M$ with $M=\prod_{i=1}^K p_i^{n_i}$. Suppose that $\Phi_{p_i^{n_i}}|A$ for some $i\in\{1,2,\dots,K\}$. Is it always true that at least one of the following must hold?

\smallskip
(i) {\em (Subgroup tiling)} $A\subset\ZZ_{M/p_i}$.

\smallskip
(ii) {\em (Obstruction)} There exists an element $x\in \ZZ_M\setminus A$ such that $A_x\subset\ell_i(x)$.
Furthermore, the pair $(A,B)$ has a $(\calp_A,\calp_B)$-cofibered structure of depth at least 2, with $1\in\calp_A$.

\end{question}

It is possible that, at least for $K\geq 4$, more complicated obstructions may occur that cannot be reduced to 1-dimensional saturating spaces. However, the results of \cite{LaLo2} show that the answer is affirmative if $M=p_1^2p_2^2p_3^2$. It seems reasonable to conjecture the following.

\begin{conjecture}\label{conj-reductions}
The answers to Questions \ref{q-slab} (both versions) and \ref{q-obstruction} are affirmative when $M$ has at most 3 distinct prime factors.
\end{conjecture}

Proposition \ref{replacement} also relates the (T2) property to divisor sets, in the sense that $B$ satisfies (T2) if any only if its divisor set $\Div(B)$ is disjoint from $\Div(A^\flat)$. The equivalence between conditions (ii) and (iii) in Theorem \ref{subtile} establishes a more granular result in this direction, by connecting a smaller family of differences in $A$ to the corresponding family of cyclotomic polynomials. Specifically, if we write $S_i:=\{m:p_i^{n_i}|m|M\}$ for a fixed $i\in\{1,\dots,K\}$, then Theorem \ref{subtile} establishes a fundamental connection between the set of differences $m\in \Div(A)\cap S_i$ and the collection of all cyclotomic polynomials $\Phi_d$ dividing $A$, for $d\in S_i$.
It would be interesting to know whether such relationships exist on the level of individual differences and cyclotomic polynomials. 
As an extreme example of a hypothetical result of this type, we state the following.

\begin{conjecture}\label{conjecture-one-divisor}
If $\Phi_{p_i^{n_i}}|A$, then $M/p_i\notin \Div(B)$. 
\end{conjecture}
This conjecture can be stated purely in terms of differences, since $\Phi_{p_i^{n_i}}|A$ if and only if $M/p_i\in Div(A^\flat)$. It is obviously necessary in order for (T2) to hold, though not sufficient.
Absent a proof of (T2) in its full generality, this might be an interesting direction to explore.

Our saturating set techniques provide partial support for the conjecture, as follows.
Suppose that $M/p_i\in\Div(B)$, with $(b-b',M)=M/p_i$ for some $b,b'\in B$. Suppose further that Conjecture
\ref {conjecture-one-divisor} is true. Then $\Phi_{p_i^{n_i}}\nmid A$, so that $\Phi_{p_i^{n_i}}|B$. In particular, 
$$
|B\cap \Pi(y,p_i^{n_i})|= \frac{1}{p_i} |B\cap \Pi(b,p_i^{n_i-1})|
$$
for every $y\in\ZZ_M$ with $(y-b,M)=M/p_i$. 
We do not know how to prove this, but we can prove the weaker statement that $B\cap \Pi(y,p_i^{n_i})\neq\emptyset$ for each such $y$. Indeed, if $y\in B$, this is obvious. If, on the other hand, $y\not\in B$, then by (\ref{bispan}) we have $B_y\subset \Pi(y,p_i^{n_i})$, and in particular $B\cap \Pi(y,p_i^{n_i})$ is nonempty.

We note that fibering plays a significant role in all our tiling arguments. For instance, if $A$ is $M$-fibered in some direction, this is sufficient to apply the slab reduction (see Remark \ref{fiberingimpliesslab}). At the other extreme, if $\Phi_M|A$ but $A$ fails to be $M$-fibered on some $D(M)$ grid, our strategy in \cite{LaLo2} is to identify and use cofibered structures, which in particular implies fibering in $B$ on a lower scale. Motivated by this, we conjecture the following.

\begin{conjecture}
For every $i$ there exists $1\leq \alpha_i < n_i$ such that either $A$ or $B$ is $M/p_i^{\alpha_i}$-fibered in the $p_i$ direction. In particular, if $M$ has 3 prime factors, $\Phi_M|A$, and there exists a $D(M)$ grid $\Lambda$ such that $A\cap\Lambda$ is not fibered in any direction, then $B$ is fibered in all directions on some scale. (This happens e.g., in Szab\'o's examples.)
\end{conjecture}

\subsection{Saturating sets}

We have seen in Lemma \ref{1dim_sat} and Corollary \ref{1dim_sat-cor} that if $A_x\subset\ell_i(x)$ for some $i\in\{1,\dots,K\}$ and $x\in \ZZ\setminus A$, this implies a cofibered structure in $(A,B)$. By Lemma \ref{fibershift}, this allows us to shift $M$-fibers in $A$ in the given direction. We use this in \cite{LaLo2} to reduce $A\oplus B=\ZZ_M$ to T2-equivalent tilings $A'\oplus B=\ZZ_M$, where $A'$ has additional regularity properties. It would therefore be interesting to either find a structure theorem (an analogue of Lemma \ref{1dim_sat}) for saturating sets contained in higher-dimensional subspaces, or, alternatively, to find a systematic way of adding geometric constraints on saturating sets until we find a cofibered structure.

In all examples where we have been able to determine saturating sets, we found that they enjoy pleasant ``splitting" properties. For example, suppose that $(x-a,M)=M/p_i$ for some $a\in A$ and $i\in\{1,\dots,K\}$. By (\ref{bispan}), we have
\begin{equation}\label{split-e1}
A_x\subset \Bispan(x,a)=\Pi(x,p_i^{n_i}) \cup\Pi(a,p_i^{n_i}).
\end{equation}
However, what actually tends to happen is that either $A_x\subset \Pi(x,p_i^{n_i})$ or $A_x\subset \Pi(a,p_i^{n_i})$. For instance, if $\bbA_{M/p_i}[a]>0$, then $\bbA_{M/p_i}[x]\geq2$ and  
$$
A_x\subset\Pi(x,p_i^{n_i}).
$$
If, however, $\bbB_{M/p_i}[b]>0$, then 
$$
A_{x,b}\subset\Pi(a,p_i^{n_i}).
$$ 
For an example of a less obvious situation where this happens, see e.g. \cite[Lemma 9.18]{LaLo2}.

Similarly, suppose that $K=3$ and let $x\in\ZZ_M$. Assume that there are two elements $a_i,a_j\in A$ such that $(x-a_i,M)=M/p_i$ and $(x-a_j,M)=M/p_j$, with $i,j\in\{1,2,3\}$ distinct.
Then, by Lemma \ref{bispan} again,
$$
A_x\subset \Bispan(x,a_i)\cap\Bispan(x,a_j)=\ell_k(x)\cup\ell_k(a_i)\cup\ell_k(a_j)\cup\ell_k(x_{ij}).
$$
where $x_{ij}\in\ZZ_M$ is the unique point such that $(x_{ij}-a_i,M)=M/p_j$ and $(x_{ij}-a_j,M)=M/p_i$. However, in all tiling examples that we have worked out, $A_x$ is in fact contained in just one of the above lines. An example of this type of situation is provided by \cite[Lemma 4.6]{LaLo2}. See also Lemma \ref{fullplanesplitting} in this paper for a different example where the initial geometric constraints  restrict the saturating set to a union of two lines (\ref{twotails}), but then additional arguments show that only one of these lines may participate.

Returning to the ``two planes" situation as in (\ref{split-e1}), we can in fact say a little bit more. 
By (\ref{split-e1}), we have for any $b\in B$,
$$
1=\sum_{p_i^{n_i}|m}\frac{1}{\phi(M/m)}(\bbA_m[x]\bbB_m[b]+\bbA_{m/p_i}[x]\bbB_{m/p_i}[b]).
$$
Suppose that $\bbB_m[b]$ and $\bbB_{m/p_i}[b]$ are both nonzero for some $m$ with $p_i^{n_i}|m|M$.
If there were an $a'\in A$ with $(x-a',M)\in\{m,m/p_i\}$, then we would also have $(a-a',M)\in\{m,m/p_i\}$, contradicting divisor exclusion. Moreover, if $\bbB_{m/p_i}[b]\neq 0$, then any $a'\in A$ with $(x-a',M)\in\{m,m/p_i\}$ must lie in the plane 
$\Pi(a,p_i^{n_i})$. Hence
\begin{equation}\label{splitplanesconj}
1=\sum_{p_i^{n_i}|m}\frac{1}{\phi(M/m)}(\delta_m\bbA_m[x]\bbB_m[b]
+(1-\delta_m)\bbA_{m}[a]\bbB_{m/p_i}[b])	
\end{equation}
where $\delta_m\in \{0,1\}$ for all $p_i^{n_i}|m$.

It appears reasonable to conjecture the following. 

\begin{conjecture}
Let $A\oplus B=\ZZ_M$ be a tiling, and assume that $(x-a,M)=M/p_i$ for some $a\in A$, $x\in\ZZ_M\setminus A$, and $i\in\{1,\dots,K\}$. Then either $A_x\subset \Pi(x,p_i^{n_i})$ or $A_x\subset \Pi(a,p_i^{n_i})$. Furthermore,
either $\delta_m=1$ for all $m$, or $\delta_m=0$ for all $m$, depending only on the choice of $a\in A$ and $b\in B$.
\end{conjecture}

A similar reasoning, with only slightly more effort, applies to $A_{x,y}$ and $B_{y,x}$ as in 
(\ref{sat-layers}) with $x\in \ZZ_M\setminus A$ and $y\in \ZZ_M\setminus B$.


\subsection{Subspace bounds}

The special case of Lemma \ref{planebound} with $K=3$ and $\alpha_i=0$ is a simple but very effective tool in \cite{LaLo2}. It would be useful to have similar bounds for lower-dimensional subspaces, for example lines in the 3-prime case. In this regard, we formulate the following modest conjecture.

\begin{conjecture}\label{linebndconj}
Suppose that $p_i^{\alpha_i}\parallel |A|$ with $\alpha_i<n_i$. Then for all $x\in\ZZ_{M}$ 
\begin{equation}
|A\cap\ell_i(x)|<p_i^{n_i}.
\end{equation}
\end{conjecture}

Conjecture \ref{linebndconj} is clearly true when:
\begin{itemize}
\item $p_i^{n_i}$ exceeds the plane bounds $|A\cap\Pi(x,p_j^{n_j})|$ for $j\neq i$, 
\item $A$ satisfies (T2) (in this case, $A\oplus B^\flat=\ZZ_M$ and there exists $\beta_i$ so that $M/p_i^{\beta_i}\notin \Div(A)$).
\end{itemize}


\section{Acknowledgement}
Both authors were supported by NSERC Discovery Grants.




\bibliographystyle{amsplain}

\begin{thebibliography}{99}

\bibitem{Bh}
S. Bhattacharya, {\it Periodicity and Decidability of Tilings of $\ZZ^2$}, Amer. J. Math. 142 (2020),
255–266.


\bibitem{CM} E. Coven, A. Meyerowitz, {\it Tiling the integers
with translates of one finite set}, J. Algebra 212 (1999),
161--174.

\bibitem{deB} N.G. de Bruijn, {\it On the factorization of cyclic groups}, Indag. Math. 15 (1953), 370--377.


\bibitem{dutkay-kraus}
D. E. Dutkay, I. Kraus, {\it On spectral sets of integers}, in: {\it Frames and Harmonic Analysis}, Contemporary Mathematics, vol. 706, edited by Y. Kim, A.K. Narayan, G. Picioroaga, and  E.S. Weber, American Mathematical Society, 2018,
pp. 215--234.




\bibitem{dutkay-lai}
D. E. Dutkay, C.-K. Lai, {\it Some reductions of the spectral set conjecture to integers}, Math. Proc. Cambridge Phil. Soc. 156 (2014), 123--135.


\bibitem{FKS} T. Fallon, G. Kiss, G. Somlai, {\it Spectral sets and tiles in $\ZZ^2_p\times \ZZ^2_q$}, preprint, arXiv:math/2105.10575


\bibitem{FMV}
T. Fallon, A. Mayeli, D. Villano, {\it The Fuglede Conjecture holds in $\mathbb{F}^3_p$ for $p=5,7$}, Proc. Amer. Math. Soc., to appear.




\bibitem{FMM}
B. Farkas, M. Matolcsi, P. M\'ora, {\it On Fuglede's conjecture and the existence of universal spectra},
J. Fourier Anal. Appl. 12(5) (2006), 483--494.

\bibitem{FR}
B. Farkas, S. G. R\'ev\'esz, {\it Tiles with no spectra in dimension 4}. Math. Scandinavica, 98 (2006), 44--52.




\bibitem{Fug} B. Fuglede, {\it Commuting self-adjoint partial differential operators
and a group-theoretic problem}, J. Funct. Anal. 16 (1974), 101--121.


\bibitem{GLW} A. Granville, I. {\L}aba, Y. Wang, {\it A characterization
of finite sets that tile the integers}, unpublished preprint, 2001, arXiv:math/0109127.

\bibitem{GL} R. Greenfeld, N. Lev {\it Fuglede's spectral set conjecture for convex polytopes}, Analysis \& PDE 10(6) (2017), 1497--1538.


\bibitem{GT} R. Greenfeld, T. Tao. {\it The structure of translational tilings in $\mathbb {Z}^ d$}, Discrete Analysis, 2021:16, 28 pp.


\bibitem{IKT}
A. Iosevich, N. Katz, T. Tao, {\it The Fuglede spectral conjecture holds for convex planar domains}, Math. Res. Lett. 10 (2003), 559--569.

\bibitem{IMP}
A. Iosevich, A. Mayeli, J. Pakianathan, {\it The Fuglede conjecture holds in $\ZZ_p \times \ZZ_p$}, Analysis \& PDE 10(4) (2017), 757--764.


\bibitem{KMSV} G. Kiss, R. D. Malikiosis, G. Somlai, M. Vizer, 
{\it On the discrete Fuglede and Pompeiu problems}, Analysis \& PDE 13 (3), (2020), 765-788.

\bibitem{KMSV2} G. Kiss, R. D. Malikiosis, G. Somlai, M. Vizer, {\it Fuglede's conjecture holds for cyclic groups of order $ pqrs$}, Discrete Anal. 2021:12, 24 pp.

\bibitem{KS} G. Kiss, S. Somlai, {\it Fuglede’s conjecture holds on $\ZZ^2_p\times\ZZ_q$}, Proc. Amer.
Math. Soc. 149 (2021), 4181--4188.


\bibitem{KoLev} M.N. Kolountzakis, N. Lev, {\it Tiling by translates of a function: results and open problems}, Discrete Analysis, 2021:12, 24 pp.

\bibitem{KM}
M.N. Kolountzakis, M. Matolcsi, {\it Complex Hadamard matrices and the spectral set conjecture}, 
Collect. Math., Vol. Extra (2006), 281--291.

\bibitem{KM2}
M. N. Kolountzakis, M. Matolcsi, {\it Tiles with no spectra}, Forum Math. 18(3) (2006), 519--528.

\bibitem{KL}
S. Konyagin, I. {\L}aba, {\it Spectra of certain types of polynomials and
tiling the integers with translates of finite sets}, J. Number Theory
103 (2003), 267--280.


\bibitem{L} I. {\L}aba, {\it The spectral set conjecture and multiplicative
properties of roots of polynomials}, J. London Math. Soc. 65 (2002), 661--671.

\bibitem{LaLo2} I. {\L}aba, I. Londner, {\it The Coven-Meyerowitz tiling conditions for 3 odd prime factors},
preprint, 2021, arXiv:2106.14044.

\bibitem{LS} J.C. Lagarias, S. Szab\'o, {\it Universal spectra and Tijdeman's
conjecture on factorization of cyclic groups}, J. Fourier Anal. Appl. 1(7) (2001), 63--70. 

\bibitem{LW1} J.C. Lagarias, Y. Wang, {\it Tiling the line with translates of
one tile}, Invent. Math. 124 (1996), 341--365.

\bibitem{LW2} J.C. Lagarias, Y. Wang, {\it Spectral sets and factorization of 
finite abelian groups}, J. Funct. Anal. 145 (1997), 73--98.


\bibitem{LL} T.Y. Lam and K.H. Leung, {\it On vanishing sums of roots of unity},
J. Algebra 224 (2000), 91--109.

\bibitem{LM} N. Lev, M. Matolcsi, {\it 	The Fuglede conjecture for convex domains is true in all dimensions}, 
arXiv:1904.12262,
to appear in Acta Mathematica.

\bibitem{M} R. D. Malikiosis, {\it On the structure of spectral and tiling subsets of cyclic groups}, preprint, arXiv:2005.05800.

\bibitem{MK} R. D. Malikiosis, M. N. Kolountzakis, {\it Fuglede's conjecture on cyclic groups of order 
$p^nq$}, Discrete Analysis, 2017:12, 16pp.

\bibitem{Mann} H. B. Mann, {\it On Linear Relations Between Roots of Unity}, Mathematika 12, Issue 2 (1965), 107--117.

\bibitem{matolcsi}
M. Matolcsi, {\it Fuglede's conjecture fails in dimension 4}, Proc. Amer. Math. Soc. 133(10) (2005), 3021--3026.

\bibitem{New} D.J. Newman, {\it Tesselation of integers}, J. Number Theory 9 (1977), 107--111.


\bibitem{Re1} L. R\'edei, {\it \"Uber das Kreisteilungspolynom}, Acta Math. Hungar. 5 (1954), 27--28.

\bibitem{Re2} L. R\'edei, {\it Nat\"urliche Basen des Kreisteilungsk\"orpers}, Abh. Math. Sem. Univ. Hamburg 23 (1959), 180--200.


\bibitem{Sands} A. Sands, {\it On Keller's conjecture for certain cyclic
groups}, Proc. Edinburgh Math. Soc. 2 (1979), 17--21.


\bibitem{schoen}
I. J. Schoenberg, {\it A note on the cyclotomic polynomial}, Mathematika 11 (1964), 131-136.

\bibitem{shi}
R. Shi, {\it Fuglede's conjecture holds on cyclic groups $\ZZ_{p^2qr}$
}, Discrete Analysis 2019:14, 14pp.

\bibitem{shi2}
R. Shi, {\it Equi-distribution on planes and spectral set conjecture on $\ZZ_{p^2}\times\ZZ_p$}, J. London Math. Soc 102(2), (2020), 1030–1046.


\bibitem{somlai} G. Somlai, {\it Spectral sets in $\ZZ_{p^2qr}$ tile}, preprint, arXiv:1907.04398.

\bibitem{Steinberger} J. P. Steinberger, {\it Minimal vanishing sums of roots of unity with large coefficients},
Proc. London Math. Soc. 97 (3) (2008), 689--717.


\bibitem{Sz} S. Szab\'o, {\it A type of factorization of finite abelian groups},
Discrete Math. 54 (1985), 121--124.

\bibitem{Szabo-book} S. Szab\'o, {\it Topics in Factorization of Abelian Groups}, Hindustan Book Agency, 2004.


\bibitem{tao-fuglede}
T. Tao, {\it Fuglede's conjecture is false in 5 and higher dimensions}, Math. Res. Letters, 11 (2004), 251--258.


\bibitem{Tao-blog} T. Tao, {\it Some notes on the Coven-Meyerowitz conjecture}, https://terrytao.wordpress.com/2011/11/19/some-notes-on-the-coven-meyerowitz-conjecture/


\bibitem{Tij} R. Tijdeman, {\it Decomposition of the integers as a direct sum of
two subsets}, in {\em Number Theory (Paris 1992--1993)}, London Math. Soc.
Lecture Note Ser., vol. 215, 261--276, Cambridge Univ. Press, Cambridge, 1995.

\bibitem{zhang} T. Zhang, {\it Fuglede's conjecture holds in $\ZZ_p\times \ZZ_{p^n}$}, preprint, arXiv:2109:08400.


\end{thebibliography}

\bigskip

\noindent{{\sc Laba:} Department of Mathematics, UBC, Vancouver,
B.C. V6T 1Z2, Canada}

\noindent{\it ilaba@math.ubc.ca}

\smallskip

\noindent{{\sc Londner:} Department of Mathematics, UBC, Vancouver,
B.C. V6T 1Z2, Canada. 
\\
{\it Current address:} Department of Mathematics, Faculty of Mathematics and Computer Science, Weizmann Institute of Science, Rehovot 7610001, Israel}

\noindent{\it itay.londner@weizmann.ac.il}

\end{document}